\documentclass[11pt]{amsart}

\usepackage{amsmath}
\usepackage{amssymb}
\usepackage{enumitem}
\usepackage{graphicx}
\usepackage{tabularx}
\usepackage{multirow}
\usepackage{hhline}

\usepackage[mathlines,modulo]{lineno}

\usepackage[colorlinks=false]{hyperref}

\usepackage{todonotes}

\usepackage{cleveref}

\theoremstyle{plain}
\newtheorem{theorem}{Theorem}[section]
\newtheorem{lemma}[theorem]{Lemma}
\newtheorem{proposition}[theorem]{Proposition}
\newtheorem{corollary}[theorem]{Corollary}

\newtheorem{conjecture}[theorem]{Conjecture}

\theoremstyle{definition}
\newtheorem{definition}[theorem]{Definition}
\newtheorem{claim}{Claim}[theorem]
\newtheorem{subclaim}{}[claim]
\newtheorem{remark}[theorem]{Remark}

\newcommand{\mcal}[1]{\ensuremath{\mathcal{#1}}}
\newcommand{\mbb}[1]{\ensuremath{\mathbb{#1}}}
\newcommand{\closure}[1]{\ensuremath{\langle #1 \rangle}}
\newcommand{\formula}[1]{\text{\fontfamily{cmss}\selectfont \textup{#1}}}

\newcommand{\dash}{\nobreakdash-\hspace{0pt}}
\newcommand{\cl}{\operatorname{cl}}
\newcommand{\bw}{\operatorname{bw}}
\newcommand{\dw}{\operatorname{dw}}
\newcommand{\ind}{\formula{Ind}}
\newcommand{\sing}{\formula{Sing}}
\newcommand{\vertex}{\operatorname{Vert}}
\newcommand{\edge}{\operatorname{Edge}}
\newcommand{\enc}{\operatorname{enc}}

\newcommand{\cmso}{\ensuremath{\mathit{CMS}_{0}}}
\newcommand{\cmstwo}{\ensuremath{\mathit{CMS}_{2}}}
\newcommand{\hgg}{$H$\dash gain-graphic}
\newcommand{\ba}{\backslash}

\title[Tree automata and matroids]
{Tree automata and pigeonhole classes of matroids: II}

\author[Funk]{Daryl Funk}
\address{Department of Mathematics,
Douglas College,
British Columbia,
Canada}
\email{funkd@douglascollege.ca}

\author[Mayhew]{Dillon Mayhew}
\address{School of Mathematics and Statistics,
Te Herenga Waka,
Victoria University of Wellington,
New Zealand}
\email{dillon.mayhew@vuw.ac.nz}

\author[Newman]{Mike Newman}
\address{Department of Mathematics and Statistics,
University of Ottawa,
Ontario,
Canada}
\email{mnewman@uottawa.ca}

\date{\today}

\begin{document}

\begin{abstract}
Let $\psi$ be a sentence in the counting monadic second-order logic
of matroids and let \mbb{F} be a finite field.
Hlin\v{e}n\'{y}'s Theorem says that we can test whether
\mbb{F}-representable matroids satisfy $\psi$ using an algorithm that is
fixed-parameter tractable with respect to branch-width.
In a previous paper we proved there is a similar
fixed-parameter tractable algorithm that can test the members of any
efficiently pigeonhole class.
In this sequel we apply results from the first paper and
thereby extend Hlin\v{e}n\'{y}'s Theorem to the classes of
fundamental transversal matroids, lattice path matroids,
bicircular matroids, and \hgg\ matroids, when $H$ is a finite group.
As a consequence, we can obtain a new proof of
Courcelle's Theorem.
\end{abstract}

\maketitle

\section{Introduction}

In the first paper of the series \cite{FMN-I}, we proved an extension of
Hlin\v{e}n\'{y}'s Theorem \cite{Hli06c}.
This theorem concerns the \emph{counting monadic second-order logic} for matroids, \cmso.
In this language we have variables $X_{1}, X_{2}, X_{3},\ldots$ representing subsets
of the ground set of a matroid.
We have a binary predicate $X_{i}\subseteq X_{j}$, which allows us to say
when one subset is contained in another.
We also have a unary independence predicate $\ind(X_{i})$, which
returns the value true when the input is an independent subset.
We further have predicates that allow us to say that a
set has cardinality $p$ modulo $q$, where $p$ and $q$ are positive integers.
Let $\psi$ be a sentence in \cmso.
Hlin\v{e}n\'{y}'s Theorem says that there is a fixed-parameter tractable
algorithm for testing whether matroids satisfy $\psi$, as long
as the input class consists of matroids representable over a finite field \mbb{F}.
A \emph{fixed-parameter tractable} algorithm
typically includes a numerical parameter, $\lambda$, and the
running time is bounded by $O(f(\lambda)n^{c})$, where $n$ is the size of the input,
$c$ is a constant, and $f(\lambda)$ is a value that depends only on $\lambda$.
In the case of Hlin\v{e}n\'{y}'s Theorem, the parameter is the
branch-width of the input matroid.
Thus the theorem provides us with a polynomial-time algorithm when
the input class is restricted to \mbb{F}\dash representable matroids
of bounded branch-width.
The main result of \cite{FMN-I} extends Hlin\v{e}n\'{y}'s Theorem.

\begin{theorem}
\label{lounge-II}
Let \mcal{M} be an efficiently pigeonhole class of matroids.
Let $\psi$ be a sentence in \cmso.
We can test whether matroids in \mcal{M} satisfy $\psi$ using an algorithm that is fixed-parameter tractable with respect to branch-width.
\end{theorem}

This sequel paper exploits \Cref{lounge-II} and related ideas to show that
there are fixed-parameter tractable algorithms for testing \cmso\ sentences
in other natural classes of matroids, beyond finite-field representable matroids.
In particular, our main theorem shows that Hlin\v{e}n\'{y}'s Theorem can be extended as
follows.

\begin{theorem}
\label{hammer}
Let \mcal{M} be any of the following:
\begin{enumerate}[label=\textup{(\roman*)}]
\item the class of fundamental transversal matroids,
\item the class of lattice path matroids,
\item the class of bicircular matroids, or
\item the class of \hgg\ matroids, where $H$ is a finite group.
\end{enumerate}
Let $\psi$ be a sentence in \cmso.
We can test whether matroids in \mcal{M} satisfy $\psi$ using an algorithm that is fixed-parameter tractable with respect to branch-width.
\end{theorem}

We now explain efficiently pigeonhole matroid classes,
along with some other associated concepts.
Formal definitions are reserved for \Cref{pigeonhole-II}.
Imagine that $M$ is a matroid, and that $U$ is a subset of $E(M)$.
We define an equivalence relation on the subsets of $U$.
Let $X$ and $X'$ be subsets of $U$.
Assume that $X\cup Z$ is independent if and only if
$X'\cup Z$ is independent, for any subset $Z\subseteq E(M)-U$.
We think of this as indicating that no subset of $E(M)-U$ can
distinguish between $X$ and $X'$.
In this case we write $X\sim_{U} X'$.
We put the elements of $E(M)$ into correspondence with the
leaves of an appropriately chosen subcubic tree.
If there are at most $q$ equivalence classes under $\sim_{U}$
whenever $U$ is a set displayed by an edge of the tree, then
the \emph{decomposition-width} of $M$ is at most $q$.
This notion of decomposition-width is equivalent to that used by
Kr\'{a}l \cite{Kra12} and by Strozecki \cite{Str10,Str11}, in the sense that the decomposition-width of a class is bounded under one definition if and only if it is bounded under the other.

A class of matroids with bounded decomposition-width must
have bounded branch-width \cite[Corollary 2.8]{FMN-I}.
The converse does not hold (\Cref{apogee}).
Let \mcal{M} be a class of matroids, and assume that every
subclass of \mcal{M} with bounded branch-width also has
bounded decomposition-width.
Then we say that \mcal{M} is a \emph{pigeonhole} class of matroids.
This is the case if and only if the dual class is pigeonhole (\cite[Corollary 5.3]{FMN-I}).
The class of \mbb{F}\dash representable matroids forms a pigeonhole
class if and only if \mbb{F} is finite (\Cref{rhesus} and \Cref{yippie}).
Fundamental transversal matroids (\Cref{yuppie}) and lattice path matroids
also form pigeonhole classes (\Cref{dinner}).

A stronger property arises quite naturally.
Imagine that \mcal{M} is a class of matroids and that $M$ is an arbitrary matroid
in \mcal{M}.
Assume that whenever $U\subseteq E(M)$ and $\lambda_{M}(U) \leq k$, then there
are at most $\pi(k)$ equivalence classes under $\sim_{U}$, where $\pi(k)$ is a value depending only on $k$.
(Recall that $\lambda_{M}(U)$ is the connectivity value of $U$ and is defined to be
$r_{M}(U)+r_{M}(V)-r(M)$.)
In this case we say that \mcal{M} is \emph{strongly pigeonhole} (\Cref{fizzer-II}).
Every strongly pigeonhole class is pigeonhole \cite[Proposition 2.11]{FMN-I},
but the converse does not hold \cite[Remark 2.12]{FMN-I}.
The class of fundamental transversal matroids is strongly
pigeonhole, and so is the class of \mbb{F}\dash representable matroids
when \mbb{F} is finite (\Cref{rhesus}).
We do not know if any of the other classes in \Cref{hammer}
are strongly pigeonhole, but we certainly believe this to be the case
(Conjectures \ref{ulster} and \ref{hubbub}).
In fact, we make the broad conjecture that
the class of matroids that are transversal and cotransversal is
a strongly pigeonhole class (\Cref{iguana}).

\Cref{lounge-II} relies on tree automata to test the sentence $\psi$,
as does Hlin\v{e}n\'{y}'s Theorem.
These machines are described in \Cref{automatic-II}.
In order to construct a parse tree for the machine to process,
we require a further strengthening of the pigeonhole property.
It is not enough that there is a bound on the number of classes
under $\sim_{U}$:
we must be able to compute this equivalence relation efficiently.
In fact, we are happy to compute a refinement of
$\sim_{U}$, as long as this refinement does not have too many
classes.
If we are able to do this, then we say that the class is
\emph{efficiently pigeonhole} (see \Cref{yakuza-II} for the formal statement).
Any efficiently pigeonhole class is also strongly pigeonhole.
The class of matroids representable over a finite field
(\Cref{rhesus}) is efficiently pigeonhole,
and this leads to a proof of Hlin\v{e}n\'{y}'s Theorem.
The class of fundamental transversal matroids
is also efficiently pigeonhole (\Cref{yuppie}).

In \cite[Theorem 6.11]{FMN-I} we also proved that there is a
fixed-parameter-tractable algorithm for testing $\psi$ in members of
\mcal{M} as long as the $3$\dash connected members of \mcal{M}
form an efficiently pigeonhole class and we can efficiently construct
descriptions of minors.
This result was motivated by the fact that we do not know if
bicircular matroids or \hgg\ matroids ($H$ finite) form
efficiently pigeonhole classes.
(We conjecture this is the case in \Cref{hubbub}.)
We have been able to show that the $3$\dash connected bicircular matroids
and the $3$\dash connected \hgg\ matroids form efficiently pigeonhole classes
(\Cref{nibble}).
This is then enough to establish cases (iii) and (iv) in \Cref{hammer}.

Knowing that we have efficient model-checking for bicircular
matroids gives us a new, and quite simple, proof of
Courcelle's Theorem (\Cref{fodder}), which states that there is
a fixed-parameter tractable algorithm for testing monadic second-order
sentences in graphs, relative to the parameter of tree-width.

As well as proving positive results, we establish some
negative propositions.
Any class of matroids that contains the rank\dash $3$ uniform matroids
and is closed under principal extensions is not pigeonhole
(\Cref{nubbin}).
Thus matroids representable over a given infinite field form
a non-pigeonhole class (\Cref{yippie}).
The class of transversal matroids is not pigeonhole,
(\Cref{hiccup}) and nor is the class of gammoids (\Cref{yellow}).
A different argument shows that
the class of \hgg\ matroids is not pigeonhole when $H$ is infinite
(\Cref{object}).

For an introduction to monadic second-order logic and its application to finite
structures, see  \cite{CE12} or \cite[Chapter 13]{DF13}.
Oxley provides our reference for the
basic concepts of matroid theory \cite{Oxl11}.
We have noted that if $M$ is a matroid and $(U,V)$ is a partition of $E(M)$, then
the \emph{connectivity value} of $U$ is
$\lambda_{M}(U)=r_{M}(U)+r_{M}(V)-r(M)$.
A \emph{$k$\dash separation} is a partition, $(U,V)$, of the
ground set such that $|U|,|V|\geq k$, and $\lambda_{M}(U)<k$.
A matroid is \emph{$n$\dash connected} if it has no
$k$\dash separations with $k<n$.
A \emph{cyclic flat} of the matroid $M$ is a flat $F$ such that the
restriction $M|F$ has no coloops.
When $E$ is a finite set and \mcal{I} is a collection of subsets of $E$, we refer to the
pair $(E,\mcal{I})$ as a \emph{set-system}.
We say that the members of \mcal{I} are \emph{independent sets}.

\section{Tree automata}
\label{automatic-II}

\begin{definition}
\label{sitcom-II}
Let $T$ be a tree with a distinguished \emph{root} vertex, $t$.
Assume that every vertex of $T$ other than $t$ has degree one or three,
and that if $T$ has more than one vertex, then $t$ has degree two (so that $T$ cannot have exactly two vertices).
The \emph{leaves} of $T$ are the degree-one vertices.
In the case that $t$ is the only vertex, we also consider $t$ to be a leaf.
Let $L(T)$ be the set of leaves of $T$.
If $T$ has more than one vertex, and $v$ is a non-leaf, then $v$ is adjacent
with two vertices that are not contained in the path from $v$ to $t$.
These two vertices are the \emph{children} of $v$.
We distinguish the \emph{left} child and the \emph{right} child of $v$.
Now let $\Sigma$ be a finite \emph{alphabet} of \emph{characters}.
Let $\sigma$ be a function from $V(T)$ to $\Sigma$.
Under these circumstances we say that $(T,\sigma)$ is a \emph{$\Sigma$\dash tree}.
\end{definition}

\begin{definition}
\label{nuance-II}
A \emph{tree automaton} is a tuple  $(\Sigma, Q, F, \delta_{0},\delta_{2})$,
where $\Sigma$ is a finite alphabet, and $Q$ is a finite set of \emph{states}.
The set of \emph{accepting states} is a subset $F\subseteq Q$.
The  \emph{transition rules}, $\delta_{0}$ and $\delta_{2}$, are partial functions
from $\Sigma$ and $\Sigma\times Q\times Q$ respectively, into $2^{Q}$,
the power set of $Q$.
\end{definition}

Let $A=(\Sigma, Q, F, \delta_{0},\delta_{2})$ be an automaton and
let $(T,\sigma)$ be a $\Sigma$\dash tree with root $t$.
We let $r\colon V(T)\to 2^{Q}$ be the function recursively defined
as follows:
\begin{enumerate}[label=\textup{(\roman*)}]
\item if $v$ is a leaf of $T$, then $r(v)$ is $\delta_{0}(\sigma(v))$
if this is defined, and is otherwise the empty set.
\item if $v$ has left child $v_{L}$ and right child $v_{R}$, then
\[
r(v)=\bigcup_{(q_{L},q_{R})\in r(v_{L})\times r(v_{R})} \delta_{2}(\sigma(v),q_{L},q_{R}),
\]
as long as the images in this union are all defined: if they are
not then we set $r(v)$ to be empty.
\end{enumerate}
We say that $r$ is the \emph{run} of the automaton $A$ on $(T,\sigma)$.
Note that we define a union over an empty collection to be the empty set.
We say that $A$ \emph{accepts} $(T,\sigma)$ if $r(t)$ contains
an accepting state.

Let $\Sigma$ be a finite alphabet, and let
$(T,\sigma)$ be a $\Sigma$\dash tree.
Let $\varphi$ be a bijection from the finite set $E$ into $L(T)$,
and let $Y$ be a subset of $E$.
We construct a
$\Sigma \cup (\Sigma\times \{0,1\})$\dash tree which we denote $\enc(T,\sigma,\varphi,Y)$.
The characters applied to the leaves of this tree will encode the
subset $Y$.
If $v$ is a non-leaf vertex of $T$, then it receives the label
$\sigma(v)$ in $\enc(T,\sigma,\varphi,Y)$.
However, if $v$ is a leaf, then it receives a label
$(\sigma(v), s)$, where $s=1$ if $\varphi^{-1}(v)$ is in $Y$ and otherwise $s=0$.

\begin{definition}
\label{option-II}
Let $\Sigma$ be a finite set and let $A$ be a tree automaton
with $\Sigma\cup (\Sigma\times \{0,1\})$ as its alphabet.
Let $(T,\sigma)$ be a $\Sigma$\dash tree, and let
$\varphi$ be a bijection from the finite set $E$ into $L(T)$.
We define the set-system $M(A,T,\sigma,\varphi)$
as follows:
\[
M(A,T,\sigma,\varphi)
=
(E,\{Y\subseteq E\colon A\ \text{accepts}\
\enc(T,\sigma,\varphi,Y)\}).
\]
\end{definition}

So $M(A,T,\sigma,\varphi)$ is the set-system consisting of all
subsets of the leaves that are accepted by the automaton.

Now let $\Sigma$ be a finite set, and let
$A$ be a tree automaton
with alphabet $\Sigma\cup (\Sigma\times \{0,1\})$.
Let $M=(E,\mcal{I})$ be a set-system.
Assume there is a $\Sigma$\dash tree $(T_{M},\sigma_{M})$, and a bijection
$\varphi_{M}\colon E\to L(T_{M})$ having the property that
$M=M(A,T_{M},\sigma_{M},\varphi_{M})$.
In this case we say that $(T_{M},\sigma_{M},\varphi_{M})$ is a \emph{parse tree} for $M$ (relative to the automaton $A$).

Note that if $(T_{M},\sigma_{M},\varphi_{M})$ is a parse tree for $M$, then
we can simulate an independence oracle for $M$ by running $A$.
We simply label each leaf $v$ with $(\sigma_{M}(v),1)$ if $\varphi_{M}^{-1}(v)$ is in $Y$ and
with $(\sigma_{M}(v),0)$ otherwise.
By then running $A$ on the resulting tree, and testing to see if it
accepts, we can check whether or not $Y$ is in \mcal{I}.
This idea is central to the proofs of Hlin\v{e}n\'{y}'s Theorem and
of \Cref{lounge-II}.

\section{Pigeonhole classes}
\label{pigeonhole-II}

This section states our central definitions: decomposition-width,
pigeonhole classes, strongly pigeonhole classes, and
efficiently pigeonhole classes.

\begin{definition}
\label{fracas-II}
Let $(E,\mcal{I})$ be a set-system, and let $U$ be a subset of $E$.
Let $X$ and $X'$ be subsets of $U$.
We say $X$ and $X'$ are \emph{equivalent} (relative to $U$),
written $X\sim_{U} X'$, if for every subset $Z\subseteq E-U$,
the set $X\cup Z$ is in \mcal{I} if and only if $X'\cup Z$ is in \mcal{I}.
\end{definition}

Clearly $\sim_{U}$ is an equivalence relation on the subsets of $U$.
No member of \mcal{I} is equivalent to a subset not in \mcal{I}.
When \mcal{I} is the set of independent sets of a matroid
(more generally, when \mcal{I} is closed under subset containment),
all dependent subsets of $U$ are equivalent.

A \emph{subcubic tree} is one in which every vertex has degree three or one.
A degree-one vertex is a \emph{leaf}.
Let $M=(E,\mcal{I})$ be a set-system.
A \emph{decomposition} of $M$ is a pair $(T,\varphi)$, where $T$ is a
subcubic tree, and $\varphi$ is a bijection from $E$ into the set of leaves of $T$.
Let $e$ be an edge joining vertices $u$ and $v$ in $T$.
Then $e$ partitions $E$ into sets $(U_{e}, V_{e})$ in the following
way: an element $x\in E$ belongs to $U_{e}$ if and only if
the path in $T$ from $\varphi(x)$ to $u$ does not contain $v$.
We say that the partition $(U_{e},V_{e})$ and the sets $U_{e}$ and
$V_{e}$ are \emph{displayed} by the edge $e$.
Define $\dw(M;T,\varphi)$ to be the maximum
number of equivalence classes in $\sim_{U}$, where the maximum is taken
over all subsets, $U$, displayed by an edge in $T$.
Define $\dw(M)$ to be the minimum value of
$\dw(M;T,\varphi)$, where the minimum is taken over all
decompositions $(T,\varphi)$ of $M$.
This value is then said to be the \emph{decomposition-width} of $M$.
If $M$ is a matroid, then $\dw(M)$ is defined to be
$\dw(E(M),\mcal{I})$.
Kr\'{a}l \cite{Kra12} and Strozecki \cite{Str10,Str11} used an
equivalent notion of decomposition-width.

Let $M$ be a matroid.
If $(T,\varphi)$ is a decomposition of $M=(E(M),\mcal{I}(M))$, then
$\bw(M;T,\varphi)$ is the maximum value of
\[\lambda_{M}(U_{e})+1=r_{M}(U_{e})+r_{M}(V_{e})-r(M)+1,\]
where the maximum is taken over all partitions $(U_{e},V_{e})$
displayed by edges of $T$.
Now the \emph{branch-width} of $M$ (written $\bw(M)$) is the
minimum value of $\bw(M;T,\varphi)$,
where the minimum is taken over all decompositions of $M$.
A class of matroids with bounded rank also has bounded branch-width.
In \cite[Corollary 2.8]{FMN-I} we show that a class of matroids with
bounded decomposition-width also has bounded branch-width.
The converse is not true (see \Cref{apogee}).
This motivates the following definition.

\begin{definition}
\label{juicer-II}
Let \mcal{M} be a class of matroids.
Then \mcal{M} is \emph{pigeonhole} if, for every positive
integer, $\lambda$, there is an integer $\rho(\lambda)$
such that $\bw(M)\leq \lambda$ implies
$\dw(M)\leq \rho(\lambda)$, for every $M\in\mcal{M}$.
\end{definition}

So a class of matroids is pigeonhole if every subclass with bounded
branch-width also has bounded decomposition-width.
The next result is \cite[Corollary 5.3]{FMN-I}.

\begin{proposition}
\label{dynamo-II}
Let \mcal{M} be a class of matroids.
Then \mcal{M} is pigeonhole if and only if
$\{M^{*}\colon M\in\mcal{M}\}$ is pigeonhole.
\end{proposition}

We often find that natural classes of matroids with the pigeonhole
property also possess a stronger property.

\begin{definition}
\label{fizzer-II}
Let \mcal{M} be a class of matroids.
Assume that for every positive integer $\lambda$, there is
a positive integer $\pi(\lambda)$, such that whenever $M\in\mcal{M}$
and $U\subseteq E(M)$ satisfies $\lambda_{M}(U)\leq \lambda$,
there are at most $\pi(\lambda)$ equivalence classes under $\sim_{U}$.
In this case we say that \mcal{M} is \emph{strongly pigeonhole}.
\end{definition}

In \cite[Proposition 2.11]{FMN-I}, we give the easy proof that any class with the
strong pigeonhole property also has the pigeonhole property.
On the other hand, the class of rank\dash $2$ matroids is pigeonhole
without being strongly pigeonhole (see \cite[Remark 2.12]{FMN-I}.)

\begin{proposition}
\label{tabard}
The class of uniform matroids is strongly pigeonhole.
\end{proposition}

\begin{proof}
Let $M$ be a rank\dash $r$ uniform matroid, and
let $U$ be a subset of $E(M)$ such that $\lambda_{M}(U)\leq \lambda$,
for some positive integer $\lambda$.
Declare subsets $X,X'\subseteq U$ to be equivalent if:
\begin{enumerate}[label=\textup{(\roman*)}]
\item $|X|, |X'| > r_{M}(U)$,
\item $r_{M}(U)-\lambda< |X|=|X'|\leq r_{M}(U)$, or
\item $|X|,|X'| \leq r_{M}(U)-\lambda$.
\end{enumerate}
Thus there are at most $\lambda+2$ equivalence classes, and
we will be done if we can show that this equivalence
relation refines $\sim_{U}$.
If $|X|,|X'|> r_{M}(U)$ then both $X$ and $X'$ are 
dependent, and hence they are equivalent under $\sim_{U}$.
Since $M$ is uniform, any subsets of $U$ with the same
cardinality will be equivalent under $\sim_{U}$.
Therefore we now need only consider the case that
$|X|,|X'|\leq r_{M}(U)-\lambda$.
Assume that $Z\subseteq E(M)-U$, and
$X\cup Z$ is independent while $X'\cup Z$ is
dependent.
Since $X'\cup Z$ is dependent, it follows that
$|X'\cup Z|>r(M)$.
As $X\cup Z$ is independent, we see that
$|Z|\leq r_{M}(E(M)-U)$.
Therefore
\[r(M)<|X'\cup Z|=|X'|+|Z|\leq r_{M}(U)-\lambda+r_{M}(E(M)-U).\]
Hence $r_{M}(U)+r_{M}(E(M)-U)-r(M)>\lambda$,
and we have a contradiction to $\lambda_{M}(U)\leq \lambda$.
\end{proof}

\Cref{lounge-II} is concerned with matroid algorithms.
For the purposes of measuring the efficiency of these algorithms,
we restrict our attention to matroid classes where there is a
succinct representation, such as graphic matroids or
finite-field-representable matroids.

\begin{definition}
\label{umpire-II}
Let \mcal{M} be a class of matroids.
A \emph{succinct representation} of \mcal{M} is a
relation, $\Delta$, from \mcal{M} into the set of finite binary strings, along
with a polynomial-time Turing Machine.
We write $\Delta(M)$ to indicate any string in the image of $M\in \mcal{M}$.
The Turing machine must, when given any input $(\Delta(M),X)$ where
$M\in\mcal{M}$ and $X$ is a subset of $E(M)$, correctly answer the question
``Is $X$ independent in $M$?".
\end{definition}

If the Turing Machine operates in time bounded by at most $p(n)$ on any input
of length $n$, where $p$ is some polynomial, then it follows that
the length of the string $\Delta(M)$ is no more than $p(|E(M)|)$.
A graph provides a succinct representation of a graphic matroid, and
a matrix provides a succinct representation of a finite-field-representable matroid.

\begin{proposition}
\label{nature}
Let $\Delta$ be a succinct representation of \mcal{M}, a class of matroids.
For each positive integer $\lambda$ let $\mcal{M}_{\lambda}$ be
$\{M\in\mcal{M}\colon \bw(M)\leq \lambda\}$.
Let $f$ be a function from positive integers to positive integers.
Assume there is an algorithm which when given a positive integer $\lambda$
will produce a tree automaton $A_{\lambda}$ in time bounded by $f(\lambda)$.
Assume also there is an algorithm which when given $(\Delta(M),\lambda)$
for $M\in\mcal{M}_{\lambda}$  will produce a parse tree for 
$M$ relative to $A_{\lambda}$ in time bounded by $f(\lambda)|\Delta(M)|^{c}$,
where $c$ is a constant.
Let $\psi$ be a sentence in \cmso.
We can test whether matroids in \mcal{M} satisfy $\psi$ using an algorithm that is fixed-parameter tractable with respect to branch-width.
Furthermore, \mcal{M} is pigeonhole.
\end{proposition}

\begin{proof}
Assume we are given the input $(\Delta(M),\lambda)$, where $M$ is in
$\mcal{M}_{\lambda}$.
We first construct $A_{\lambda}$ in time bounded by $f(\lambda)$.
Assume that $(T_{M},\sigma_{M},\varphi_{M})$ is the parse tree that we construct for
$M\in\mcal{M}_{\lambda}$ relative to $A_{\lambda}$.
This takes time bounded by $f(\lambda)|\Delta(M)|^{c}$.
Now \cite[Lemma 4.7]{FMN-I} tells us that there is a tree automaton
$A_{\lambda,\psi}$ such that $A_{\lambda,\psi}$ accepts $(T_{M},\sigma_{M},\varphi_{M})$
if and only $M$ satisfies $\psi$.
Moreover, the proof of this lemma is constructive, and it tells us how
to build $A_{\lambda,\psi}$, given $A_{\lambda}$.
The time taken to build $A_{\lambda,\psi}$ depends only $\lambda$ and $\psi$
(which we regard as fixed).
Let this time be equal to $g_{\psi}(\lambda)$.

Applying $A_{\lambda,\psi}$ to $(T_{M},\sigma_{M},\varphi_{M})$ takes time bounded
by $h_{\psi}(\lambda)|\Delta(M)|$, for some value $h_{\psi}(\lambda)$ that
depends only on $\psi$ and $\lambda$.
Now the total process of testing $M$ to see if it satisfies $\psi$ takes time
bounded by
\[f(\lambda)(|\Delta(M)|^{c}+1)+g_{\psi}(\lambda)+h_{\psi}(\lambda)|\Delta(M)|.\]
This establishes the existence of the fixed-parameter tractable algorithm.

The existence of $A_{\lambda}$ and $(T_{M},\sigma_{M},\varphi_{M})$ for any $M\in\mcal{M}_{\lambda}$
means that $\mcal{M}_{\lambda}$ is \emph{automatic} (using the language of \cite{FMN-I}).
Because it is automatic, it has bounded decomposition-width
(\cite[Lemma 5.1]{FMN-I}).
So for any positive integer $\lambda$, the class $\mcal{M}_{\lambda}$ has bounded decomposition-width.
This is exactly what it means for \mcal{M} to be pigeonhole.
\end{proof}

\begin{definition}
\label{toffee-II}
Let $\Delta$ be a succinct representation of \mcal{M}, a
class of matroids.
We say that $\Delta$ is \emph{minor-compatible}
if there is a polynomial-time algorithm which will accept
any tuple $(\Delta(M),X,Y)$ when $M\in\mcal{M}$ and $X$ and
$Y$ are disjoint subsets of $E(M)$, and return
a string of the form $\Delta(M/X\ba Y)$.
\end{definition}

In order to construct parse trees for automata to process,
we need to be able to efficiently compute the equivalence classes of $\sim_{U}$.
In fact, we are happy to compute an equivalence relation that refines $\sim_{U}$,
as long as it does not have too many classes.

\begin{definition}
\label{yakuza-II}
Let \mcal{M} be a class of matroids with a succinct representation $\Delta$.
Assume there is a Turing Machine, a constant $c$, and a function $\pi$ from positive integers to positive integers such that the machine takes as input any tuple $(\Delta(M),U,X,X',\lambda)$, where $M$ is in \mcal{M}, $U\subseteq E(M)$
satisfies $\lambda_{M}(U)\leq \lambda$, and $X$ and $X'$ are subsets of $U$.
Assume also that in time bounded by $O(\pi(\lambda)|E(M)|^{c})$, the machine computes an
equivalence relation, $\approx_{U}$, on the subsets of $U$, so that it
accepts $(\Delta(M),U,X,X',\lambda)$ if and only if $X\approx_{U} X'$.
Furthermore we assume that
\begin{enumerate}[label=\textup{(\roman*)}]
\item $X\approx_{U} X'$ implies $X\sim_{U} X'$, and
\item the number of equivalence classes under $\approx_{U}$ is at most $\pi(\lambda)$.
\end{enumerate}
Under these circumstances, we say that \mcal{M} is
\emph{efficiently pigeonhole} (relative to $\Delta$).
\end{definition}

\begin{remark}
The proof of \Cref{lounge-II} essentially follows from \Cref{nature} and the
observation that given $\Delta(M)$ we can construct a branch-decomposition
for $M$ in polynomial time (\cite[Proposition 6.3]{FMN-I}).
We can then convert this decomposition tree into a parse tree.
The automaton which processes this tree can also be constructed in
time bounded by $O(\lambda^{c})$.
\Cref{lounge-II} follows\footnote{Note that the proof of \Cref{lounge-II} requires
a definition of efficiently pigeonhole classes that is very slightly stronger than the one
presented in \cite{FMN-I}.}.
\end{remark}

Clearly an efficiently pigeonhole class of matroids is
also strongly pigeonhole.
In \Cref{yuppie}, we will prove that the class of fundamental transversal matroids
is efficiently pigeonhole.
Statement (i) of \Cref{hammer} will then immediately follow, by an application of
\Cref{lounge-II}.

\section{Non-pigeonhole classes}
\label{labour}

Next we develop some tools for proving negative results.
We want to certify that certain classes are not pigeonhole.
Recall that a matroid with rank $r$ is \emph{sparse paving} if
every circuit has cardinality either $r$ or $r+1$ and when
$C$ and $C'$ are distinct circuits of size $r$ then $|C\cap C'|<r-1$.
Let $G$ be a simple graph with edge set $\{e_{1},\ldots, e_{m}\}$
and vertex set $\{v_{1},\ldots, v_{n}\}$, where $n\geq 3$.
We define $m(G)$ to be the rank\dash $3$ sparse
paving matroid with ground set
$\{v_{1},\ldots, v_{n}\}\cup\{e_{1},\ldots, e_{m}\}$.
The only non-spanning circuits of $m(G)$ are the
sets $\{v_{i},e_{k},v_{j}\}$, where $e_{k}$ is an
edge of $G$ joining the vertices $v_{i}$ and $v_{j}$.

\begin{lemma}
\label{apogee}
Let \mcal{M} be a class of matroids.
Assume there are arbitrarily large integers, $N$, such that \mcal{M}
contains a matroid isomorphic to $m(K_{N})$.
Then \mcal{M} contains rank\dash $3$ matroids with arbitrarily high
decomposition-width.
Hence \mcal{M} is not pigeonhole.
\end{lemma}

\begin{proof}
Observe that rank\dash $3$ matroids have branch-width at most four,
so if $\{M\in\mcal{M}\colon r(M)= 3\}$ has unbounded decomposition-width,
then \mcal{M} is certainly not pigeonhole.
Assume for a contradiction that every rank\dash $3$ matroid in \mcal{M}
has decomposition-width at most $K$.

Let $k$ be an arbitrary integer greater than $K$, and 
let $N$ be an integer such that
\[
\frac{1}{3}\left(N+\binom{N}{2}\right) \geq 7k^{2}+2k.
\]
and \mcal{M} contains a matroid, $M$, isomorphic to $m(K_{N})$.
By relabelling, we assume that
the ground set of $M$ is
$
\{v_{1},\ldots, v_{N}\}\cup\{e_{ij}\colon 1\leq i<j\leq N\}
$
and the only non-spanning circuits are
of the form $\{v_{i},e_{ij},v_{j}\}$.
Let $(T,\varphi)$ be a decomposition of $M$ with the
property that if $U$ is any displayed set, then
$\sim_{U}$ has at most $K$ classes.
Using \cite[Lemma 14.2.2]{Oxl11}, we choose
an edge $e$ in $T$ such that each of the
displayed sets, $U_{e}$ and $V_{e}$, contains at least
\[
\frac{1}{3}|E(M)|=\frac{1}{3}\left(N+\binom{N}{2}\right)\geq 7k^{2}+2k
\]
elements.
Let $G$ be a complete graph with vertex set
$\{v_{1},\ldots, v_{N}\}$ and edge set
$\{e_{ij}\colon 1\leq i<j\leq N\}$, where $e_{ij}$
joins $v_{i}$ to $v_{j}$.
We colour a vertex or edge red if it belongs to
$U_{e}$, and blue if it belongs to $V_{e}$.

Assume that there are at least $2k$ red vertices
and at least $2k$ blue vertices.
Then we can assume without loss of generality that there
is a matching in $G$ consisting of $k$ red
edges, each of which joins a red vertex to a blue vertex.
Thus we can find elements
$v_{i_{1}},\ldots, v_{i_{k}}$ in $U_{e}$ and elements
$v_{j_{1}},\ldots, v_{j_{k}}$ in $V_{e}$ such that
$e_{i_{p}j_{p}}$ is in $U_{e}$ for each $p$.
If $p$ and $q$ are distinct, then
$\{v_{i_{p}},e_{i_{p}j_{p}},v_{j_{p}}\}$ is a circuit of $M$ while
$\{v_{i_{q}},e_{i_{q}j_{q}},v_{j_{p}}\}$ is a basis.
Hence $\{v_{i_{p}},e_{i_{p}j_{p}}\}$ and
$\{v_{i_{q}},e_{i_{q}j_{q}}\}$ are inequivalent under $\sim_{U_{e}}$.
This means that $\sim_{U_{e}}$ has at least
$k$ equivalence classes.
As $k>K$, this is a contradiction, so
we assume without loss of generality that
there are fewer than $2k$ red vertices.

Assume some red vertex is joined to at least $k$ blue vertices
by red edges.
Then there is an element $v_{i}\in U_{e}$ and
elements $v_{j_{1}},\ldots, v_{j_{k}}\in V_{e}$ such that
$e_{ij_{p}}$ is in $U_{e}$ for each $p$.
For distinct $p$ and $q$, we see that
$\{v_{i},e_{ij_{p}},v_{j_{p}}\}$ is a circuit while
$\{v_{i},e_{ij_{q}},v_{j_{p}}\}$ is a basis.
Therefore $\{v_{i},e_{ij_{p}}\}$ and
$\{v_{i},e_{ij_{q}}\}$ are inequivalent under $\sim_{U_{e}}$.
We again reach the contradiction that there are at least $k$
equivalence classes under $\sim_{U_{e}}$.
Now we can deduce that there are fewer than
$2k^{2}$ red edges that join a red vertex to a blue vertex.

There are fewer than $2k$ red vertices and fewer than
\[
\binom{2k}{2}<4k^{2}
\]
red edges that join two red vertices.
Since the number of red edges and vertices is at least
one third of $N+\binom{N}{2}$, we see that the number of
red edges joining two blue vertices is at least
\[
\frac{1}{3}\left(N+\binom{N}{2}\right)-
(2k+2k^{2}+4k^{2})\geq k^{2}.
\]
A result of Abbott, Hanson, and Sauer \cite{AHS72}
implies that the subgraph induced by red edges that join two blue
vertices contains either a vertex of degree at least $k$,
or a matching containing at least $k$ edges.
In the former case, there are elements
$v_{i},v_{j_{1}},\ldots, v_{j_{k}}\in V_{e}$
such that $e_{ij_{p}}$ is in $U_{e}$ for each $p$.
Then $\{v_{i},e_{ij_{p}},v_{j_{p}}\}$ is a circuit,
and $\{v_{i},e_{ij_{p}},v_{j_{q}}\}$ is a basis
for distinct $p$ and $q$, so
$\{v_{i},v_{j_{p}}\}$ and
$\{v_{i},v_{j_{q}}\}$ are inequivalent under $\sim_{V_{e}}$.
This leads to a contradiction, so there is a matching
of at least $k$ edges.
Therefore we can find elements
$v_{i_{1}},\ldots, v_{i_{k}},v_{j_{1}},\ldots, v_{j_{k}}$ in
$V_{e}$ such that each $e_{i_{p}j_{p}}$ is in $U_{e}$.
For distinct $p$ and $q$, we see that
$\{v_{i_{p}},e_{i_{p}j_{p}},v_{j_{p}}\}$ is a circuit and
$\{v_{i_{q}},e_{i_{p}j_{p}},v_{j_{q}}\}$ is a basis.
Therefore $\{v_{i_{p}},v_{j_{p}}\}$ and
$\{v_{i_{q}},v_{j_{q}}\}$ are inequivalent under
$\sim_{V_{e}}$, so we reach a final contradiction that
completes the proof.
\end{proof}

Let $F$ be a flat of the matroid $M$.
Let $M'$ be a single-element extension of $M$,
and let $e$ be the element in $E(M')-E(M)$.
We say that $M'$ is a \emph{principal extension} of $M$
(relative to $F$) if $F\cup e$ is a flat of $M'$ and whenever
$X\subseteq E(M)$ spans $e$ in $M'$, it spans $F\cup e$.

\begin{corollary}
\label{nubbin}
Let \mcal{M} be a class of matroids.
If \mcal{M} contains all rank\dash $3$ uniform matroids,
and is closed under principal extensions, then it is
not pigeonhole.
\end{corollary}

\begin{proof}
We note that $m(K_{N})$ can be constructed by starting with a
rank\dash $3$ uniform matroid, the elements
of which represent the vertices of $K_{N}$.
The elements representing edges are then added via
principal extensions.
The result now follows from \Cref{apogee}.
\end{proof}

\section{Representable matroids}

The next result is not surprising, and is implicitly utilised in the work of
Hlin\v{e}n\'{y} \cite{Hli06c} and Kr\'{a}l \cite{Kra12}.

\begin{theorem}
\label{rhesus}
Let \mbb{F} be a finite field.
The class of \mbb{F}\dash representable matroids
is efficiently pigeonhole.
\end{theorem}

\begin{proof}
Assume that $|\mbb{F}|=q$.
Let \mcal{M} be the class of \mbb{F}\dash representable matroids.
We consider the succinct representation $\Delta$ that
sends each matroid in \mcal{M} to an \mbb{F}\dash matrix representing it.
Let $M$ be a rank\dash $r$ matroid in \mcal{M}, and let $U$ be a subset of $M$.
We use $V$ to denote $E(M)-U$.
We identify $M$ with a multiset of points in the projective
geometry $P=\mathrm{PG}(r-1,q)$
(we lose no generality in assuming that $M$ is loopless).
If $X$ is a subset of $E(M)$, then \closure{X} will denote
its closure in $P$.

Assume that $\lambda_{M}(U)\leq \lambda$.
Grassman's identity tells us that
the rank of $\closure{U} \cap \closure{V}$ is equal to
$r(U)+r(V)-r(M)\leq \lambda$.
We define the equivalence relation $\approx_{U}$ so that if
$X$ and $X'$ are subsets of $U$, then
$X\approx_{U} X'$ if both $X$ and $X'$ are dependent, or
both are independent and
$\closure{X}\cap\closure{V}=\closure{X'}\cap\closure{V}$.
Deciding whether $X\approx_{U} X'$ is true requires only
elementary linear algebra, and it can certainly be accomplished
in time bounded by $O(|E(M)|^{c})$ for some constant $c$.
Since $\closure{U} \cap \closure{V}$ is a subspace of $P$ with
affine dimension at most $\lambda-1$,
it contains at most $(q^{\lambda}-1)/(q-1)$ points.
Therefore  $2^{q^{\lambda-1}+q^{\lambda-2}+\cdots+1}+1$
is a crude upper bound on the number of
$(\approx_{U})$\dash classes.
It remains only to show that $\approx_{U}$ refines $\sim_{U}$.

Assume that $X\approx_{U} X'$, and yet
$X\cup Z$ is independent while $X'\cup Z$ is dependent,
where $Z$ is a subset of $V$.
Then $X$ is independent, so $X'$ is independent also.
Let $C$ be a circuit contained in $X'\cup Z$.
As both $X'$ and $Z$ are independent, neither
$X'\cap C$ nor $Z\cap C$ is empty.
Now the rank of $\closure{X'\cap C} \cap \closure{Z\cap C}$
is
\[
r(X'\cap C)+r(Z\cap C)-r(C)=|X'\cap C|+|Z\cap C|-(|C|-1)=1.
\]
Let $c$ be the point of $P$ that is in
$\closure{X'\cap C} \cap \closure{Z\cap C}$.
Since $c$ is in $\closure{X'}\cap \closure{V}$, our assumption tells us it is
also in $\closure{X}\cap \closure{V}$.

Assume $c$ is not in $X$.
Since it is in $\closure{X}$, we can let $C_{X}$ be a circuit
contained in $X\cup c$ that contains $c$.
If $c$ is in $Z$, then $X\cup Z$ contains $C_{X}$, and we have
a contradiction, so $c$ is not in $Z$.
We let $C_{Z}$ be a circuit contained in $Z\cup c$ that contains $c$.
Circuit elimination between $C_{X}$ and $C_{Z}$ shows that
$X\cup Z$ contains a circuit, and again we have a contradiction.
Therefore $c$ is in $X$.
If $c$ is not in $Z$, then $Z\cup c\subseteq X\cup Z$ contains a
circuit.
Therefore $c$ is in $Z$.
As $X$ and $Z$ are disjoint subsets of $E(M)$,
but $c$ is identified with elements of both, we conclude that
$M$ contains a parallel pair, with one element in $X$, and
the other in $Z$.
Again $X\cup Z$ is dependent, and we have a final contradiction.
\end{proof}

Hlin\v{e}n\'{y}'s Theorem \cite{Hli06c} follows immediately from
\Cref{lounge-II,rhesus}.
We note that proofs of Hlin\v{e}n\'{y}'s Theorem can also be derived from the
works by Kr\'{a}l \cite{Kra12} and Strozecki \cite{Str11}.

\begin{proposition}
\label{yippie}
Let \mbb{K} be an infinite field.
Then the class of \mbb{K}\dash representable matroids
is not pigeonhole.
\end{proposition}

\begin{proof}
This follows almost immediately from \Cref{nubbin}
and \cite[Lemma 2.1]{MNW09}.
\end{proof}

\section{Fundamental transversal matroids}

Transversal matroids can be thought of geometrically
as those obtained by placing points freely on the faces of a simplex.
A transversal matroid is \emph{fundamental} if there is a point
placed on each vertex of that simplex.
More formally, a transversal matroid is fundamental if it has a basis,
$B$, such that $r(B\cap Z)=r(Z)$, for every cyclic flat $Z$ (see \cite{BKM11}).
In this case $B$ is a basis consisting of points located at the vertices of the simplex.
From this characterisation it is easy to see that the dual of a fundamental
transversal matroid is also fundamental.

In this, and subsequent sections, we will show that three subclasses
of transversal matroids are pigeonhole: fundamental transversal matroids (\Cref{yuppie}),
lattice path matroids (\Cref{copeck}), and bicircular matroids (\Cref{nibble}).
To start with, we prove that we cannot extend these results to the entire
class of transversal matroids.

\begin{proposition}
\label{hiccup}
The class of transversal matroids is not pigeonhole.
\end{proposition}

\begin{proof}
By \Cref{dynamo-II}, we can prove that the class of transversal matroids
is not pigeonhole by proving the same statement for the class of
cotransversal matroids.
Certainly this class contains all rank\dash $3$ uniform matroids.
Recall that the matroid $M$ is cotransversal if and only if it is a
\emph{strict gammoid} \cite{IP73}.
This means that there is a directed graph $G$ with vertex set $E(M)$,
and a distinguished set, $T$, of vertices, where
$X\subseteq E(M)$ is independent in $M$ if and only if
there are $|X|$ vertex-disjoint directed paths, each of them starting
with a vertex in $X$ and terminating with a vertex in $T$.
Assume that $G$ is such a directed graph, and that $F$ is a flat of $M$.
Create the graph $G'$ by adding the new vertex $e$, and
arcs directed from $e$ to each of the vertices in $F$.
It is an easy exercise to verify that if $M'$ is the strict
gammoid corresponding to $G'$, then $M'$ is a principal
extension of $M$ by $F$.
This demonstrates that the class of cotransversal matroids is
closed under principal extensions, so the \namecref{hiccup}
follows by \Cref{nubbin}.
\end{proof}

\begin{remark}
\label{yellow}
From \Cref{hiccup} we see that any class of matroids
containing transversal matroids is not pigeonhole.
In particular, the class of gammoids is not pigeonhole.
\end{remark}

Let $G$ be a bipartite graph, with bipartition $A\cup B$.
There is a fundamental transversal matroid, $M[G]$, with
$A\cup B$ as its ground set, where $X\subseteq A\cup B$ is
independent if and only if there is a matching, $W$, of $G$ such that
$|W|=|X\cap A|$ and each edge in $W$ joins a vertex in $X\cap A$
to a vertex in $B-X$.
In this case we say that $W$ \emph{certifies} $X$ to be independent.
This definition implies that $B$ is a basis of $M[G]$,
and $r(B\cap Z)=r(Z)$ for any cyclic flat $Z$.
Moreover, any fundamental transversal matroid can be represented in this way.
Note that we can represent $M[G]$ with a standard bipartite presentation
by adding an auxiliary vertex, $b'$, for each vertex $b\in B$, and
making $b'$ adjacent only to $b$.
We then swap the labels on $b$ and $b'$.
The transversal matroid on the ground set $A\cup B$
represented by this bipartite graph is equal to $M[G]$.

\begin{theorem}
\label{yuppie}
The class of fundamental transversal matroids is efficiently
pigeonhole.
\end{theorem}

\begin{proof}
We consider the succinct representation of fundamental transversal
matroids that involves representing such a matroid with a
bipartite graph.
Let $M=M[G]$ be a fundamental transversal matroid, where $A\cup B$
is a bipartition of the bipartite graph $G$, and $B$ is a basis of $M$.
Let $(U,V)$ be a partition of $A\cup B$, and assume that
$\lambda_{M}(U)\leq \lambda$.
Our goal will be to construct an equivalence relation $\approx$ on the subsets of
$U$ such that $\approx$ satisfies the conditions in \Cref{yakuza-II}.

Let $H$ be the subgraph of $G$ induced by edges that join
vertices in $A\cap U$ to vertices in $B\cap V$, and vertices
in $A\cap V$ to vertices in $B\cap U$.

\begin{claim}
\label{hiatus}
Any matching of $H$ contains at most $\lambda$ edges.
\end{claim}

\begin{proof}
Let $W$ be a matching in $H$.
Let $A_{U}$ and $A_{V}$, respectively, be the set of vertices
in $A\cap U$ (respectively $A\cap V$) that are incident with
an edge in $W$.
Therefore $|A_{U}|+|A_{V}|=|W|$.
If we restrict $W$ to edges incident with vertices in $A\cap U$,
then it certifies that $(B\cap U)\cup A_{U}$ is an
independent subset of $U$.
Similarly, $(B\cap V)\cup A_{V}$ is an
independent subset of $V$.
Therefore
\[
\lambda\geq
r(U)+r(V)-r(M)
\geq
|B\cap U| +|A_{U}|
+|B\cap V| +|A_{V}|
-|B|
=|W|.\qedhere
\]
\end{proof}

We can find a maximum matching of $H$,
using one of a number of polynomial-time algorithms.
It follows from K\H{o}nig's Theorem \cite{Kon31} that $H$
contains a vertex cover, $S$, such that $|S|\leq \lambda$.
Furthermore, K\H{o}nig's Theorem is constructive: given a
maximum matching of $H$, we can find $S$ in polynomial time.
From this point onwards, we regard $S$ as being fixed.

Let $X$ be an independent subset of $U$, and let $W$ be a
matching that certifies its independence.
We will construct a \emph{signature}, $\mcal{C}(X,W)$.
Signatures of subsets of $V$ will be defined symmetrically at the same time,
so in fact we let $\{P,Q\}$ be $\{U,V\}$, and we let $X$ be an
independent subset of $P$, with $W$ a matching certifying the
independence of $X$.
Recall that this means that $|W|=|X\cap A|$ and each edge of $W$
joins a vertex in $X\cap A$ to a vertex in $B-X$.

The signature $\mcal{C}(X,W)$ is a sequence
$(S_{1},\mcal{S}_{2},S_{3},S_{4})$, where
$S_{1}$, $S_{3}$, and $S_{4}$ are subsets of $B\cap P\cap S$,
$A\cap P\cap S$, and $B\cap Q\cap S$, respectively, and
$\mcal{S}_{2}$ is a collection of subsets of $A\cap Q\cap S$.
We define $\mcal{C}(X,W)$ as follows.
\begin{enumerate}[label=\textup{(\roman*)}]
\item $S_{1}$ is the set of vertices in $B\cap P\cap S$ that are
either in $X$ or incident with an edge in $W$.
\item A subset $Z\subseteq A\cap Q\cap S$ is in $\mcal{S}_{2}$
if and only if there is a matching $W'$ satisfying
$W\subseteq W'$ and $|W'-W|=|Z|$, where each edge in $W'-W$
joins a vertex in $Z$ to a vertex in $(B\cap P)-(S\cup X)$.
Note that $\mcal{S}_{2}$ is closed under subset inclusion.
\item $S_{3}$ is the set of vertices in $A\cap P\cap S$ that are
joined by an edge of $W$ to a vertex in $(B\cap Q)-S$.
Note that any vertex in $S_{3}$ is in $X$, since it is in $A$ and incident with
an edge of $W$.
\item $S_{4}$ is the set of vertices in $B\cap Q\cap S$
that are incident with an edge in $W$.
\end{enumerate}
\begin{figure}[htb]
	\centering
	\includegraphics[scale=1.1]{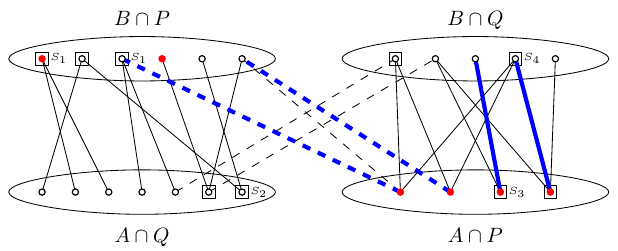}
	\caption{Defining a signature.}
	\label{fig2}
\end{figure}
We illustrate these definitions in \Cref{fig2}.
This shows a graph, $G$, with bipartition $A\cup B$,
and a partition, $(P,Q)$, of $A\cup B$.
The edges not in $H$ cross the diagram diagonally, and are drawn with
dashed lines, while the unbroken edges are the edges of $H$.
In this example the vertex cover, $S$, contains nine
vertices, which are marked with squares.
Observe that every edge of $H$ is incident with a vertex in $S$.
The set $X\subseteq P$ is marked by filled vertices.
Its independence is certified by the matching $W$,
which is drawn with heavy lines.
Vertices in the sets $S_{1}$, $S_{3}$, and $S_{4}$ are marked.
The sets in the family $\mcal{S}_{2}$ are the empty set
and the singleton set containing the vertex marked $S_{2}$.

For any independent subset $X\subseteq U$, let $\mcal{C}(X)$ be the set
\[\{\mcal{C}(X,W)\colon W\ \text{is a matching certifying that}\ X\
\text{is independent}\}.\]
Now we define the equivalence relation $\approx_{U}$.
If $X$ and $X'$ are subsets of $U$, then say that
$X\approx_{U} X'$ if both $X$ and $X'$ are dependent,
or both are independent and $\mcal{C}(X)=\mcal{C}(X')$.
Thus independent sets $X$ and $X'$ are equivalent under $\approx_{U}$
if they have exactly the same signatures.

We can choose a signature by choosing three subsets of $S$, and a family of
subsets of $S$.
Recalling that $|S| \leq \lambda$, we see that the number of signatures is at
most $(2^{\lambda})^{3}\cdot 2^{2^{\lambda}} = 2^{3\lambda+2^{\lambda}}$.
Therefore the number of $(\approx_{U})$\dash classes is no more than
\begin{equation}
\label{pi-bound}
\pi(\lambda) = 2^{2^{3\lambda+2^{\lambda}}}+1.
\end{equation}
Now we have established that condition (ii) in \Cref{yakuza-II} holds, where the
function $\pi(\lambda)$ is defined by \eqref{pi-bound}.
To complete the proof that the class of fundamental transversal matroids
is efficiently pigeonhole, we must establish that we can compute the
equivalence relation in time bounded by $O(\pi(\lambda)|E(M)|^{c})$ and that
$\approx_{U}$ refines $\sim_{U}$.
Let us first consider the task of computing $\approx_{U}$.
Recall that $\{P,Q\} = \{U,V\}$.

\begin{claim}
\label{yabber}
Let $X$ be an independent subset of $P$.
Let $(S_{1},Z,S_{3},S_{4})$ be a sequence of sets from
$B\cap P\cap S$, $A\cap Q\cap S$, $A\cap P\cap S$, and $B\cap Q\cap S$.
We can test in polynomial time whether there is a matching $W$,
certifying the independence of $X$, such that
$\mcal{C}(X,W)=(S_{1},\mcal{S}_{2},S_{3},S_{4})$ where
$Z$ is in $\mcal{S}_{2}$.
\end{claim}

\begin{proof}
To start with, we check that $S_{1}$ contains $X\cap B\cap P\cap S$ and that
$S_{3}$ is contained in $X$.
If this is not the case, then we halt and return the answer NO, so now
we assume that $X\cap B\cap P\cap S\subseteq S_{1}$ and
$S_{3} \subseteq X$.

Our strategy involves constructing an auxiliary graph, $G'$, by deleting
vertices and edges from $G$.
The construction of $G'$ is best described by the diagram in \Cref{fig8}.
Any vertex not shown in this diagram is deleted in the construction of $G'$.
Thus from $B\cap P$ we delete any vertex in $X$ and any vertex in
$S-S_{1}$.
From $A\cap Q$ we delete any vertex not in $Z$.
From $A\cap P$ we delete those vertices not in $X$.
Note that the assumption in the first paragraph of this proof means that
we have not deleted any vertex in $S_{3}$.
In $B\cap Q$, we delete those vertices in $S-S_{4}$.

\begin{figure}[htb]
\centering
\includegraphics[scale=1.1]{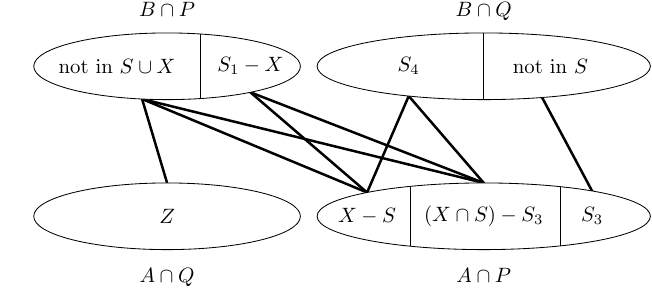}
\caption{The construction of $G'$.}
\label{fig8}
\end{figure}

Next we delete any edge of $G$ that is not represented by an edge in \Cref{fig8}.
For example, we delete any edge joining a vertex in $S_{3}$ to a vertex outside
of $(B\cap Q)-S$.
Obviously $G'$ can be constructed in time bounded by $O(|E(M)|^{c})$,
for some constant $c$.

\begin{subclaim}
\label{beacon}
The following statements are equivalent:
\begin{enumerate}[label=\textup{(\roman*)}]
\item There is a matching, $W$, of $G$, such that $W$ certifies the independence of
$X$ and $\mcal{C}(X,W)=(S_{1},\mcal{S}_{2},S_{3},S_{4})$ with $Z\in \mcal{S}_{2}$.
\item $G'$ has a matching incident with every vertex in $(X\cup Z)\cap A$
and $(S_{1}-X)\cup S_{4}$.
\end{enumerate}
\end{subclaim}

\begin{proof}
Assume (i) holds.
Then $|W|=|X\cap A|$, and every vertex in $X\cap A$ is incident with an edge of $W$.
Furthermore, every edge of $W$ joins a vertex in $X\cap A$ to a vertex in $B-X$.
Since $Z$ is in $\mcal{S}_{2}$, we can let $W'$ be a matching such that
$W\subseteq W'$, $|W'-W|=|Z|$, and each edge of $W'-W$ joins a vertex in $Z$ to
a vertex in $(B\cap P)-(S\cup X)$.
We will show that every edge of $W'$ is in $G'$, and hence $W'$ is a matching of
$G'$.

Let $ab\in W'$ be an edge joining $a\in A$ to $b\in B$.
Certainly every edge of $W'-W$ is an edge of $G'$.
Therefore we will assume that $ab$ is in $W$.
Then $a$ is in $X$, and hence in $P$.
Note that $b$ is not in $X$, for no edge of $W$ joins two vertices in $X$.
Assume that $a$ is in $X-S$.
If $b$ is in $P$, then it is either in $(B\cap P)-(S\cup X)$, or
it is in $S$.
In the latter case, $b$ is in $S_{1}$ by the definition of $S_{1}$.
In either case, $ab$ is an edge of $G'$.
Now assume $b$ is in $Q$.
Then $b$ must be in $S$, or else $ab$ is an edge of
$H$ not incident with a vertex in $S$.
In this case $b$ is in $S_{4}$ by definition, so $ab$ is an edge of $G'$.

Next assume $a$ is in $(X\cap S)-S_{3}$.
Assume $b$ is in $P$.
Then either $b$ is in $(B\cap P)-(S\cup X)$, or it is in $S_{1}-X$.
In either case $ab$ is an edge of $G'$.
Now assume $b$ is in $Q$.
If $b$ is not in $S$, then the definition of $S_{3}$ implies $a\in S_{3}$, since $a$ is in $S$.
But this contradicts the assumption $a\in (X\cap S)-S_{3}$.
Therefore $b$ is in $S$.
This means $b$ is in $S_{4}$, so $ab$ is an edge of $G'$.

Finally, we assume that $a$ is in $S_{3}$.
Then the definition of $S_{3}$ means that $b$ is in $(B\cap Q)-S$, and again $ab$ is an edge of $G'$.
Now we have shown that $W'$ is a matching of $G'$.
Every vertex in $(X\cup Z)\cap A$ is incident with an edge in $W'$, and the same statement
is true for vertices in $(S_{1}-X)\cup S_{4}$, by the definitions of $S_{1}$
and $S_{4}$.
Therefore (ii) holds.

Now assume (ii) holds.
Let $W'$ be a matching of $G'$ such that each vertex in $(X\cup Z)\cap A$ or
$(S_{1}-X)\cup S_{4}$ is incident with an edge of $W'$.
Let $W$ be the set of edges in $W'$ incident with vertices in $X\cap A$.
There is no vertex in $X\cap B$ contained in $G'$, so it immediately follows that
in $G$, the matching $W$ certifies the independence of $X$.
Every vertex in $S_{1}-X$ is incident with an edge of $W$, and no vertex of
$(B\cap P \cap S)-S_{1}$ is (since these vertices are not in $G$).
Therefore in $\mcal{C}(X,W)$, the first entry is $S_{1}$, as desired.
Every edge of $W'-W$ joins a vertex in $Z$ to a vertex in
$(B\cap P)-(S\cup X)$, so $W'$ certifies that $Z$ belongs to the second entry of
$\mcal{C}(X,W)$.
Because $S_{3}$ is a subset of $X$ we see that any vertex in $S_{3}$
is matched by $W$ to a vertex in $(B\cap Q)-S$.
Furthermore no vertex in $(X\cap S\cap A)-S_{3}$ is, by
the construction of $G'$.
Therefore the third entry of $\mcal{C}(X,W)$ is equal to $S_{3}$.
Finally, every vertex in $S_{4}$ is incident with an edge in $W$, and
no vertex of $(B\cap Q\cap S)-S_{4}$ is (since these vertices are not in $G'$).
Therefore $\mcal{C}(X,W)=(S_{1},\mcal{S}_{2},S_{3},S_{4})$, where $Z$ is in
$\mcal{S}_{2}$, so (i) holds.
\end{proof}

Now we complete the proof of \Cref{yabber}.
To test whether $W$ exists, we find a maximum-sized matching of $G'$,
using standard methods.
If this matching is incident with all the vertices in $(X\cup Z)\cap A$
(and is thus complete), then we continue, otherwise we return NO.
It is easy to see that we can use alternating-path methods to test
whether there is a complete matching that matches all the vertices in
$(S_{1}-X)\cup S_{4}$ as well as those in $(X\cup Z)\cap A$.
We return YES if such a complete matching exists, and NO otherwise,
observing that \Cref{beacon} justifies the correctness of this algorithm.
\end{proof}

To test whether $X\approx_{U} X'$, we first test whether $X$ and $X'$
are independent.
We can certainly test this in polynomial-time via a standard matching algorithm.
Assuming both $X$ and $X'$ are independent, we simply go through each
possible certificate, and check that each certificate belongs to
$\mcal{C}(X)$ if and only if it belongs to $\mcal{C}(X')$.
By using \Cref{yabber}, we can accomplish this
in time bounded by $O(\pi(\lambda)|E(M)|^{c})$, for some
constant $c$, where $\pi(\lambda)$ is the function in \eqref{pi-bound}.

Now our final task in the proof of \Cref{yuppie} is to show that
$\approx_{U}$ refines $\sim_{U}$.
To this end, we assume that $X\subseteq U$ and $Y\subseteq V$
are independent subsets of $M$.
Let $S_{X}=(S_{1},\mcal{S}_{2},S_{3},S_{4})$ be a signature in
$\mcal{C}(X)$, and let
$T_{Y}=(T_{1},\mcal{T}_{2},T_{3},T_{4})$ be a member of $\mcal{C}(Y)$.
Note that $S_{1}, T_{4}\subseteq B\cap U$ and $S_{4}, T_{1}\subseteq B\cap V$,
while $\mcal{S}_{2}$ is a family of subsets of $A\cap V$ and $\mcal{T}_{2}$
is a family of subsets of $A\cap U$.
We also have $S_{3}\subseteq A\cap U$ and $T_{3}\subseteq A\cap V$.
We declare $S_{X}$ and $T_{Y}$ to be \emph{compatible} if the
following conditions hold:
\begin{enumerate}[label=\textup{(\roman*)}]
\item $S_{1}\cap T_{4}=\emptyset$,
\item $T_{3}\in \mcal{S}_{2}$,
\item $S_{3}\in \mcal{T}_{2}$, and
\item $S_{4}\cap T_{1}=\emptyset$.
\end{enumerate}

We will prove that $X\cup Y$ is independent in $M$ if and only if
we can find signatures in $\mcal{C}(X)$ and $\mcal{C}(Y)$ that are
compatible.
This task is completed in \Cref{utopia}, and its converse (\Cref{dialer}).
From this it will easily follow that $\approx_{U}$ refines $\sim_{U}$, and
that therefore the class of fundamental transversal matroids is
efficiently pigeonhole.

\begin{claim}
\label{utopia}
Let $X\subseteq U$ and $Y\subseteq V$ be independent subsets of
$M$.
If $X\cup Y$ is independent in $M$ then there are
signatures $S_{X}\in\mcal{C}(X)$ and $T_{Y}\in\mcal{C}(Y)$ such that
$S_{X}$ and $T_{Y}$ are compatible.
\end{claim}

\begin{proof}
Let $W$ be a matching certifying that $X\cup Y$ is independent.
Thus $|W| = |X\cup Y|$, and every edge of $W$ joins a vertex in $(X\cup Y)\cap A$
to a vertex in $B-(X\cup Y)$.
Let $W_{X}$ and $W_{Y}$ be the subsets of $W$ consisting of
edges incident with vertices in $X$ (respectively $Y$).
Then $W_{X}$ certifies the independence of $X$, and $W_{Y}$
certifies the independence of $Y$.
We assert that the signatures $\mcal{C}(X,W_{X})$ and
$\mcal{C}(Y,W_{Y})$ are compatible.

Let $\mcal{C}(X,W_{X})$ be $(S_{1},\mcal{S}_{2},S_{3},S_{4})$
and let $\mcal{C}(Y,W_{Y})$ be $(T_{1},\mcal{T}_{2},T_{3},T_{4})$.
Then $S_{1}$ is the set of vertices in $B\cap U\cap S$ that are
either in $X$ or incident with an edge of $W_{X}$.
On other hand, $T_{4}$ is the set of vertices in $B\cap U\cap S$
incident with an edge of $W_{Y}$.
No edge in $W_{Y}$ is incident with an edge in $W_{X}$, or with a
vertex in $B\cap X$, so it is clear that $S_{1}$ and $T_{4}$ are disjoint.
Similarly, $S_{4}$ is the set of vertices in $B\cap V\cap S$ that are incident
with an edge of $W_{X}$ and $T_{1}$ is the set of vertices in $B\cap V\cap S$
that are either in $Y$, or incident with a vertex in $W_{Y}$.
This implies that $S_{4}\cap T_{1}=\emptyset$.

Note that $T_{3}$ is the set of vertices in $A\cap V\cap S$ that are
joined by an edge of $W_{Y}$ to a vertex in $(B\cap U)-S$.
Let $W'$ be the union of $W_{X}$ along with the set of edges in
$W_{Y}$ that are incident with a vertex in $T_{3}$.
Clearly $W'$ is a matching as it is a subset of $W$.
Also, $W_{X}\subseteq W'$ and $|W'-W_{X}|=|T_{3}|$.
Each edge in $W'-W_{X}$ is incident with a vertex in $T_{3}$,
and with a vertex in $(B\cap U)-S$.
Furthermore, no such edge is incident with a vertex in $X$,
since edges of $W$ join vertices in $(X\cup Y)\cap A$ to vertices
in $B-(X\cup Y)$.
Therefore each edge in $W'-W_{X}$ joins a vertex of $T_{3}$ to one
in $(B\cap U)-(S\cup X)$.
We have established that $T_{3}$ is contained in $\mcal{S}_{2}$.
The symmetrical argument shows that $S_{3}$ is in $\mcal{T}_{2}$.
Therefore $\mcal{C}(X,W_{X})$ and $\mcal{C}(Y,W_{Y})$ are compatible,
as we claimed.
\end{proof}

\begin{claim}
\label{dialer}
Let $X\subseteq U$ and $Y\subseteq V$ be independent subsets of
$M$.
If there are signatures $S_{X}\in\mcal{C}(X)$ and $T_{Y}\in\mcal{C}(Y)$
such that $S_{X}$ and $T_{Y}$ are compatible, then
$X\cup Y$ is independent in $M$.
\end{claim}

\begin{proof}
We assume that
$\mcal{C}(X,W_{X})=(S_{1},\mcal{S}_{2},S_{3},S_{4})$
and
$\mcal{C}(Y,W_{Y})=(T_{1},\mcal{T}_{2},T_{3},T_{4})$
are compatible signatures.
Recall that $S_{1}$ and $T_{4}$ are subsets of $B\cap U$ and
that $S_{4}$ and $T_{1}$ are subsets of $B\cap V$.
Furthermore $T_{3}$ is a subset of $A\cap V$ and $\mcal{S}_{2}$ is a
family of subsets of $A\cap V$.
Finally, $S_{3}$ is a subset of $A\cap U$ and $\mcal{T}_{2}$ is a family of subsets
of $A\cap U$.
We will construct a matching that certifies the independence
of $X\cup Y$.

Recall that $S_{3}$ is the subset of $A\cap U\cap S$ containing
vertices that are joined by edges of $W_{X}$ to vertices in
$(B\cap V)-S$.
Let $W_{X}''$ be the subset of $W_{X}$ containing edges
that are incident with vertices in $S_{3}$.
The compatibility of the signatures means that $S_{3}$ is in $\mcal{T}_{2}$.
Hence there is a matching, $W_{Y}'$, such that $W_{Y}\subseteq W_{Y}'$,
$|W_{Y}'-W_{Y}|=|S_{3}|$, and each edge of $W_{Y}'-W_{Y}$
joins a vertex in $S_{3}$ to one in $(B\cap V)-(S\cup Y)$.

Similarly, we let $W_{Y}''$ be the subset of $W_{Y}$ containing
edges that are incident with vertices in $T_{3}$.
Thus each edge in $W_{Y}''$ joins a vertex in $T_{3}$ to a vertex
in $(B\cap U)-S$.
As $T_{3}$ is in $\mcal{S}_{2}$, we can let $W_{X}'$ be a matching
such that $W_{X}\subseteq W_{X}'$,
$|W_{X}'-W_{X}|=|T_{3}|$, and each edge of $W_{X}'-W_{X}$
joins a vertex in $T_{3}$ to a vertex in $(B\cap U)-(S\cup X)$.
We now make the definition
\[
W=(W_{X}'-W_{X}'')\cup (W_{Y}'-W_{Y}'').
\]
We will prove that $W$ is a matching certifying the independence of
$X\cup Y$.

\begin{subclaim}
\label{noodle}
$W$ is a matching.
\end{subclaim}

\begin{proof}
Note that $W_{X}'-W_{X}''$ and $W_{Y}'-W_{Y}''$ are certainly matchings.
So if $W$ is not matching then there is a vertex $w$, and distinct
edges $wx\in W_{X}'-W_{X}''$ and $wy\in W_{Y}'-W_{Y}''$.

We first assume that $w$ is in $A$.
Assume also that $w$ is in $U$.
Because $Y$ is a subset of $V$, no edge of $W_{Y}$ is incident with a
vertex in $A\cap U$.
Therefore $wy$ is in $W_{Y}'-W_{Y}$.
This means that $wy$ joins a vertex of $S_{3}$
to a vertex in $(B\cap V)-(S\cup Y)$.
In particular this means that $w$ is in $S_{3}$.
No edge in $W_{X}'-W_{X}$ is incident with a vertex in $A\cap U$,
so $wx$ is not in $W_{X}'-W_{X}$.
Therefore it is in $W_{X}$, so
$wx$ is an edge of $W_{X}$ that is incident with a vertex in
$S_{3}$.
But this means that $wx$ is in $W_{X}''$, contradicting $wx\in W_{X}'-W_{X}''$.
If $w$ is in $V$, then we reach the similar
contradiction that $wy$ is in $W_{Y}''$.
Therefore we must now assume that $w$ is in $B$.

We assume that $w$ is in $B\cap U$.
No edge of $W_{Y}'-W_{Y}$  is incident with a vertex in $B\cap U$.
Therefore $wy$ is in $W_{Y}$, and it follows that $y$ belongs to $A\cap V$.

If $w$ is not in $S$, then $y$ is in $S$, for otherwise $wy$ is an edge
of $H$ that is not incident with the vertex cover $S$.
But in this case, $wy$ joins an edge of $A\cap V\cap S$ to a vertex in
$(B\cap U)-S$.
This implies that $y$ is in $T_{3}$.
Now $wy$ belongs to $W_{Y}''$, and we have a contradiction to $wy\in W_{Y}'-W_{Y}''$.
Thus $w$ is in $S$.

If $wx$ is in $W_{X}'-W_{X}$, then $wx$ joins a vertex in
$T_{3}$ to a vertex in $(B\cap U)-(S\cup X)$.
This is impossible, as we have already confirmed that $w$ is in $S$.
Hence $wx$ is in $W_{X}$.
Now $w$ is in $S$ and is incident with an edge of $W_{X}$, meaning that it is in $S_{1}$.
Furthermore, the edge $wy$ is in $W_{Y}$, and this means that $w$ is in $T_{4}$.
Thus $S_{1}\cap T_{4}\ne \emptyset$, and we have a contradiction
to the fact that
$\mcal{C}(X,W_{X})$ and $\mcal{C}(Y,W_{Y})$ are compatible.

If $w$ is in $V$, then we reach the symmetric contradiction that
either $wx$ is in $W_{X}''$, or $w$ is in $S_{4}\cap T_{1}$.
This completes the proof that $W$ is a matching.
\end{proof}

Now we must demonstrate that the matching $W$ certifies the independence
of $X\cup Y$.
First we show that every edge of $W$ joins a vertex in $(X\cup Y)\cap A$
to a vertex in $B-(X\cup Y)$.
To this end, we let $ab$ be an edge of $W$, where $a$ is in $A$ and $b$ is in $B$.
Without loss of generality we can assume that $ab$ is in $W_{X}'-W_{X}''$.
Then $a$ is either in $X$ or in $T_{3}$, which is a subset of $Y$.
Therefore $a$ is in $(X\cup Y)\cap A$.

We demonstrate that $b$ is not in $X\cup Y$.
If $ab$ is in $W_{X}'-W_{X}$, then $ab$ joins a vertex in $T_{3}$
to a vertex in $(B\cap U)-(S\cup X)$.
In this case, $b$ is certainly not in $X$.
Since $b$ is in $B\cap U$, and $Y\subseteq V$, it follows that $b$ is
also not in $Y$.
Therefore we will now assume that $ab$ is not in $W_{X}'-W_{X}$, so it is in $W_{X}$.
Each edge of $W_{X}$ joins a vertex of
$A\cap X$ to a vertex of $B-X$, so $b$ is not in $X$.
Assume that $b$ is in $Y$, so that it belongs to $B\cap V$.
Since $X\subseteq U$, it follows that $a$ is in $A\cap U$.
If $b$ is not in $S$, then $a$ is in $S$, for otherwise $ab$
is an edge of $H$ that is not incident with the vertex cover $S$.
In this case $ab$ is an edge of $W_{X}$ joining a vertex in
$A\cap U\cap S$ to a vertex in $(B\cap V)-S$, so $a$ is in
$S_{3}$, and $ab$ is in $W_{X}''$, a contradiction.
Therefore $b$ is in $S$.
As $b$ is in $B\cap Y\cap S$, it follows that it is in $T_{1}$.
But the edge $ab$ also certifies that $b$ is in $S_{4}$.
Therefore $S_{4}\cap T_{1}\ne \emptyset$, and we have
a contradiction to the fact that
$\mcal{C}(X,W_{X})$ and $\mcal{C}(Y,W_{Y})$ are compatible.
We have shown that $b$ is not in $X\cup Y$.
Hence every edge of $W$ joins a vertex of
$A\cap (X\cup Y)$ to a vertex in $B-(X\cup Y)$.

In the final step we must show that every vertex of $(X\cup Y)\cap A$ is incident
with an edge in $W$.
Let $a$ be a vertex in $(X\cup Y)\cap A$.
Without loss of generality, we will assume that $a$ is in $A\cap X$.
Then $a$ is certainly incident with an edge of $W_{X}$.
If it is not incident with an edge of $W$, then it is not incident with
an edge in $W_{X}'-W_{X}''$.
It follows that in this case, $a$ is incident with an edge of $W_{X}''$.
This implies $a$ is in $S_{3}$.
But in this case $w$ is incident with an edge of
$W_{Y}'-W_{Y}$, and hence with an edge of $W$.
Thus any vertex of $A\cap (X\cup Y)$
is incident with an edge in $W$, so $W$ certifies the
independence of $X\cup Y$, exactly as we desired.
\end{proof}

Let $X$ and $X'$ be two independent subsets of $U$ such that
$X\approx_{U} X'$.
Then $\mcal{C}(X)=\mcal{C}(X')$.
Let $Y\subseteq V$ be an independent set such that
$X\cup Y$ is independent.
\Cref{utopia} shows that there are compatible signatures
$S_{X}\in\mcal{C}(X)$ and $T_{Y}\in\mcal{C}(Y)$.
As $S_{X}$ is also in $\mcal{C}(X')$ it follows from \Cref{dialer} that
$X'\cup Y$ is independent.
This implies that $X\sim_{U}X'$, so $\approx_{U}$ refines $\sim_{U}$,
as desired.
Now the proof of \Cref{yuppie} is complete.
\end{proof}

Case (i) in \Cref{hammer} follows immediately from \Cref{yuppie} and
\Cref{lounge-II}.

\begin{remark}
\label{lubber}
\Cref{yuppie} shows that although a class of matroids may be
strongly pigeonhole, its minor-closure may not even be pigeonhole.
We can deduce this from \Cref{yellow} because the smallest
minor-closed class containing the fundamental transversal
matroids is the class of gammoids.
\end{remark}

\section{Lattice path matroids}

The class of \emph{lattice path matroids} was introduced by
Bonin, de Mier, and Noy \cite{BdMN03}.
It is closed under duality and minors \cite[Theorems 3.1 and 3.5]{BdMN03}.
Every lattice path matroid is transversal.
We describe an algorithm that constructs a parse tree for a
given lattice path matroid (\Cref{copeck}).
When combined with \Cref{nature} this shows that there is a
fixed-parameter tractable algorithm for testing \cmso\ sentences in
lattice path matroids.
It also shows that the class is pigeonhole.

A lattice path matroid is represented by a pair of strings
made from the alphabet $\{E,N\}$.
Let $P$ be such a string, so that $P = p_{1}p_{2}\cdots p_{s}$,
where each $p_{i}$ is equal to either $E$ or $N$.
When $i$ is in $\{1,\ldots, s\}$, we let $n_{i}(P)$ stand for the number of $N$\dash characters
in $\{p_{1},\ldots, p_{i}\}$.
We also let $N(P)$ stand for $\{i\in \{1,\ldots, s\}\colon p_{i}=N\}$.
Thus $n_{s}(P) = |N(P)|$.
If $P$ and $Q$ are strings of $s$ characters from the alphabet $\{E,N\}$ then we write $P\preccurlyeq Q$ to mean that $n_{i}(P) \leq n_{i}(Q)$ for each $i\in\{1,\ldots, s\}$.

Now assume that $P$ and $Q$ each contain $r$ copies of $N$ and $m$ copies of $E$.
Any such string can be identified with a path in the integer lattice
from $(0,0)$ to $(m,r)$ that uses only North and East steps.
If $P\preccurlyeq Q$ then the path $P$ never goes above $Q$.
In this case, an \emph{intermediate string}, $L$, is composed of $r$ copies of $N$
and $m$ copies of $E$ and satisfies $P\preccurlyeq L\preccurlyeq Q$.
Note that $P$ and $Q$ are both intermediate strings.

Let $P$ and $Q$ be strings composed of $r$ copies of $N$ and $m$
copies of $E$ such that $P\preccurlyeq Q$.
The \emph{lattice path matroid} $M[P,Q]$ has $\{1,\ldots, m+r\}$ as its
ground set.
The family of bases of $M[P,Q]$ is $\{N(L)\colon L\ \text{is an intermediate string}\}$.

As we have mentioned, every minor of a lattice path matroid is a lattice path matroid.
We now give an explicit description of such minors, following \cite[p.~707]{BdM06}.
Imagine that $M=M[P,Q]$ is a lattice path matroid, where $P$ and $Q$
contain $r$ copies of $N$ and $m$ copies of $E$.
Let $P$ be $p_{1}p_{2}\cdots p_{m+r}$ and let $Q$ be
$q_{1}q_{2}\cdots q_{m+r}$.
Let $i$ be in $\{1,\ldots, m+r\}$.
Assume that $n_{i}(P) = n_{i}(Q)$.
Then $i$ is a coloop of $M$ if $p_{i}=q_{i}=N$ and $i$ is a loop if
$p_{i}=q_{i}=E$.
In either of these cases, both $M\ba i$ and $M/i$ are equal to
\[M[p_{1}p_{2}\cdots p_{i-1}p_{i+1}\cdots p_{m+r}, q_{1}q_{2}\cdots q_{i-1}q_{i+1}\cdots q_{m+r}]\]
after we relabel the elements $i+1, i+2,\ldots, m+r$ in $M\ba i$ or $M/i$ as $i,i+1,\ldots, m+r-1$.

Now we assume that neither of these scenarios applies.
Let $j$ be the largest integer in $\{1,\ldots, i\}$ such that $p_{j}=E$, and let $k$ be the
least integer in $\{i,\ldots, m+r\}$ such that $q_{k}=E$.
Then 
\[M\ba i = M[p_{1}p_{2}\cdots p_{j-1}p_{j+1}\cdots p_{m+r}, q_{1}q_{2}\cdots q_{k-1}q_{k+1}\cdots q_{m+r}],\]
where we apply exactly the same relabelling as before.
Next let $j$ be the least integer in $\{i,\ldots, m+r\}$ such that $p_{j}=N$, and let $k$ be the
largest integer in $\{1,\ldots, i\}$ such that $q_{k}=N$.
In this case
\[M/ i = M[p_{1}p_{2}\cdots p_{j-1}p_{j+1}\cdots p_{m+r}, q_{1}q_{2}\cdots q_{k-1}q_{k+1}\cdots q_{m+r}].\]

\begin{proposition}
\label{throat}
Let $M=M[P,Q]$ be a lattice path matroid, where $P$ and $Q$ contain $r$ copies of $N$
and $m$ copies of $E$.
If $n_{i}(Q)-n_{i}(P) \geq t$ for some $i\in \{1,\ldots, m+r\}$, then $M$ contains a
minor isomorphic to $U_{t,2t}$.
\end{proposition}

\begin{proof}
Let $P$ be $p_{1}\cdots p_{m+r}$ and let $Q$ be $q_{1}\cdots q_{m+r}$.
Assume that the result fails, so that $M$ has no $U_{t,2t}$\dash minor.
Furthermore assume that $M$ has been chosen subject to this failure
so that $m+r$ is as small as possible.

If $p_{1}=N$, then $q_{1}=N$ because otherwise $P\preccurlyeq Q$ fails.
In this case $M/1$ is isomorphic to $M[p_{2}\cdots p_{m+r},q_{2}\cdots q_{m+r}]$.
But now $q_{2},\ldots, q_{i}$ contains at least $t$ more copies of $N$ than
$p_{2},\ldots, p_{i}$.
Since $M/1$ has no $U_{t,2t}$\dash minor, the minimality of $M$ has been
contradicted.
Therefore $p_{1}=E$.
Symmetrically, if $q_{m+r}=N$, then $p_{m+r}=N$ or else
$n_{m+r-1}(P)>n_{m+r-1}(Q)$.
In this case $M/(m+r)$ is isomorphic to $M[p_{1}\cdots p_{m+r-1},q_{1}\cdots q_{m+r-1}]$.
But if $i < m+r$ then the first $i$ characters of $p_{1}\cdots p_{m+r-1}$ are the
same as the first $i$ characters of $p_{1}\cdots p_{m+r}$, and the same applies to
$q_{1}\cdots q_{m+r-1}$.
If $i=m+r$, then $q_{1}\cdots q_{m+r-1}$ contains at least $t$ more
$N$\dash characters than $p_{1}\cdots p_{m+r-1}$.
In either case we see that $M/(m+r)$ is a smaller counterexample, so
$q_{m+r}=E$.

Let $k\in\{1,\ldots, m+r\}$ be the least integer such that $q_{k}=E$, and let
$Q'=q_{1}\cdots q_{k-1}q_{k+1}\cdots q_{m+r}$.
Let $P'$ be $p_{2}\cdots p_{m+r	}$.
Thus $M\ba 1$ is isomorphic to $M[P',Q']$.
Now $n_{i-1}(P') = n_{i}(P)$.
If $k \leq i$, then $n_{i-1}(Q')=n_{i}(Q)$, and in this case
$n_{i-1}(Q')-n_{i-1}(P') \geq t$, so $M\ba 1$ provides a smaller counterexample.
Therefore $k > i$, which means that $q_{1},\ldots, q_{i}$ are all equal to $N$.

Assume there is at least one $N$\dash character in $p_{1},\ldots, p_{i}$.
Let $j\in\{1,\ldots, i\}$ be the least integer such that $p_{j}=N$.
Let $Q'' = q_{2}\cdots q_{m+r}$ and let $P''$ be
$p_{1}\cdots p_{j-1}p_{j+1}\cdots p_{m+r}$.
Thus $M/1$ is isomorphic to $M[P'',Q'']$.
In this case
\[
n_{i-1}(Q'') - n_{i-1}(P'') = (i-1)-(n_{i}(P)-1) = i-n_{i}(P) = n_{i}(Q)-n_{i}(P) \geq t
\]
so $M/1$ is a smaller counterexample.
Hence $p_{1},\ldots, p_{i}$ are all equal to $E$.

Note that $i \geq t$.
Assume that $i > t$.
Then
\[
n_{i-1}(Q') - n_{i-1}(P') = (i-1)-0\geq t,
\]
which implies that $M\ba 1$ is a smaller counterexample.
We conclude that $i=t$.

Assume that there is an integer $j\in \{i+1,\ldots, m+r\}$ such that
$p_{j}=E$, and let $j$ be the largest such integer.
We have noted that $q_{m+r}=E$.
Now we define $P'$ to be $p_{1}\cdots p_{j-1}p_{j+1}\cdots p_{m+r}$ and
$Q'$ to be $q_{1}\cdots q_{m+r-1}$.
Thus $M\ba (m+r)$ is isomorphic to $M[P',Q']$.
But $n_{i}(Q') = n_{i}(Q)=t$ and  $n_{i}(P') = n_{i}(P)=0$, so we have a
smaller counterexample.
Therefore $p_{i+1},\ldots, p_{m+r}$ are all equal to $N$.

Note that $m+r \geq 2t$.
Assume that $m+r>2t$.
Since $p_{1},\ldots, p_{i}$ are $E$\dash characters and
$p_{i+1},\ldots, p_{m+r}$ are $N$\dash characters, and $i=t$, this
means that $P$ and $Q$ contain more than $t$ copies of $N$.
Let $k$ be the largest integer such that $q_{k}=N$.
Since $Q$ contains more than $i=t$ copies of $N$, we see that $k > i$.
Define $Q''$ to be $q_{1}\cdots q_{k-1}q_{k+1}\cdots q_{m+r}$ and define
$P''$ to be $p_{1}\cdots p_{m+r-1}$.
Then $M/(m+r)$ is isomorphic to $M[P'',Q'']$.
Again we see that $n_{i}(Q'')=n_{i}(Q)=t$ and $n_{i}(P'') = n_{i}(P)=0$.
Now we have a smaller counterexample, so $m+r=2t$.

We have shown that $P$ consists of $t$ copies of $E$ followed by $t$ copies of $N$.
The first $t$ characters in $Q$ are $N$, so the following $t$ characters must
be $E$.
It is now very easy to see that $M$ is isomorphic to $U_{t,2t}$, so we have a final
contradiction.
\end{proof}

\begin{corollary}
\label{jargon}
Let $M=M[P,Q]$ be a lattice path matroid, where $P$ and $Q$ contain $r$ copies of $N$
and $m$ copies of $E$.
Assume that $\bw(M) \leq \lambda$.
Then $n_{i}(Q)-n_{i}(P) \leq \frac{1}{2}(3\lambda - 3)$ for each
$i\in\{1,\ldots, m+r\}$.
\end{corollary}

\begin{proof}
Let $t$ be equal to $n_{i}(Q)-n_{i}(P)$ for an arbitrary choice of
$i$ in $\{1,\ldots, m+r\}$.
Then \Cref{throat} implies that $M$ has an $U_{t,2t}$\dash minor.
The branch-width of this minor is
$\lceil 2t/3\rceil+1$ by \cite[Exercise 14.2.5]{Oxl11}.
Now \cite[Proposition 14.2.3]{Oxl11} implies that
$\lceil 2t/3 \rceil + 1 \leq \bw(M) \leq \lambda$.
The result follows.
\end{proof}

In the next result, we aim to apply \Cref{nature} to the class
of lattice path matroids.
Lattice path matroids are represented succinctly by pairs of strings.

\begin{theorem}
\label{copeck}
For each positive integer $\lambda$ there is a tree automaton $A_{\lambda}$ that
can be constructed in time bounded by $f(\lambda)$, for some function $f$ on the positive integers, so that when we are given $(P,Q,\lambda)$, where $M=M[P,Q]$ is a lattice path matroid with branch-width at most $\lambda$, we can construct a parse tree for $M$ relative to $A_{\lambda}$ in time bounded by $f(\lambda)|E(M)|^{c}$, where $c$ is a constant.
\end{theorem}

\begin{proof}
Let $p$ be $\lfloor \tfrac{1}{2}(3\lambda-3)\rfloor$.
Let $\mathcal{P}$ be the power set of $\{1,2,\ldots,p\}$.
Let $\Sigma$ be the set that contains the empty function $\epsilon$,
as well as all functions from $\mathcal{P}\times \{0,1\}$ or from $\{0,1\}\times\{0,1\}$ into $\mathcal{P}$.
Then $A_{\lambda}$ has $\Sigma\cup \{(\epsilon,0),(\epsilon,1)\}$ as its alphabet.
The state space of $A_{\lambda}$ is $\mathcal{P}$, and all states are accepting except for the empty set.

Assume that $P$ and $Q$
contain $r$ copies of $N$ and $m$ copies of $E$.
Thus the ground set of $M$ is $\{1,\ldots, m+r\}$.
Let $P=p_{1}p_{2}\cdots p_{m+r}$ and let $Q=q_{1}q_{2}\cdots q_{m+r}$.
Assume that $\bw(M)\leq \lambda$.
We construct the tree $T_{M}$ as shown in \Cref{fig9}.
The bijection $\varphi_{M}$ takes the element $i \in\{1,\ldots, m+r\}$
to the leaf $v_{i}$.

\begin{figure}[htb]
\centering
\includegraphics[scale=1.1]{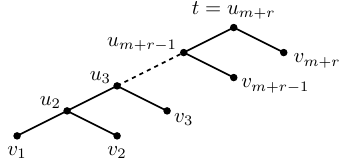}
\caption{The parse tree for a lattice path matroid.}
\label{fig9}
\end{figure}

We imagine that $Y$ is a subset of $\{1,2,\ldots,m+r\}$.
The idea behind the operation of $A_{\lambda}$ is that it will apply a member of $\mathcal{P}$ to each of the internal nodes of $T_{M}$.
The set applied to $u_{i}$ will contain all possible values $|N(L_{i})|-|N(p_{1}p_{2}\cdots p_{i})|$, where $L_{i}$ ranges over all strings of $i$ characters satisfying $p_{1}p_{2}\cdots p_{i}\preccurlyeq L_{i}\preccurlyeq q_{1}q_{2}\cdots q_{i}$ and $Y\cap \{1,2,\ldots, i\}\subseteq N(L_{i})$.
Note that \Cref{jargon} implies that such a value can be at most $p$ and that $Y$ is independent in $M$ if and only if the set applied to the root $t$ is non-empty.

To implement this idea, we describe $\sigma_{M}$, which applies a function to each node of $T_{M}$.
We say that $\sigma_{M}$ applies the empty function $\epsilon$ to any leaf.
Next we let $g\in \Sigma$ be the function applied by $\sigma_{M}$ to $u_{2}$.
The domain of $g$ is $\{0,1\}\times \{0,1\}$.
The output of $g$ depends on the characters in $p_{1}p_{2}$ and $q_{1}q_{2}$.
In the following table, the columns are labelled by the array
\[\begin{matrix}q_{1}q_{2}\\p_{1}p_{2}\end{matrix}\]
and each row shows the output of $g$ on the members of $\{0,1\}\times\{0,1\}$.

\begin{center}
\begin{tabular}{|c||m{15mm}|m{15mm}|m{15mm}|m{15mm}|}\hline
&\hfil $(0,0)$&\hfil $(0,1)$&\hfil $(1,0)$&\hfil $(1,1)$\\ \hhline{|=||=|=|=|=|}
EE&\hfil\multirow{2}{*}{$\{0\}$}&\hfil\multirow{2}{*}{$\emptyset$}&\hfil\multirow{2}{*}{$\emptyset$}&\hfil\multirow{2}{*}{$\emptyset$}\\
EE&&&&\\\hline
EN&\hfil\multirow{2}{*}{$\{0,1\}$}&\hfil\multirow{2}{*}{$\{1\}$}&\hfil\multirow{2}{*}{$\emptyset$}&\hfil\multirow{2}{*}{$\emptyset$}\\
EE&&&&\\\hline
NE&\hfil\multirow{2}{*}{$\{0,1\}$}&\hfil\multirow{2}{*}{$\{1\}$}&\hfil\multirow{2}{*}{$\{1\}$}&\hfil\multirow{2}{*}{$\emptyset$}\\
EE&&&&\\\hline
NN&\hfil\multirow{2}{*}{$\{0,1,2\}$}&\hfil\multirow{2}{*}{$\{1,2\}$}&\hfil\multirow{2}{*}{$\{1,2\}$}&\hfil\multirow{2}{*}{$\{2\}$}\\
EE&&&&\\\hline
EN&\hfil\multirow{2}{*}{$\{0\}$}&\hfil\multirow{2}{*}{$\{0\}$}&\hfil\multirow{2}{*}{$\emptyset$}&\hfil\multirow{2}{*}{$\emptyset$}\\
EN&&&&\\\hline
NE&\hfil\multirow{2}{*}{$\{0\}$}&\hfil\multirow{2}{*}{$\{0\}$}&\hfil\multirow{2}{*}{$\{0\}$}&\hfil\multirow{2}{*}{$\emptyset$}\\
EN&&&&\\\hline
NN&\hfil\multirow{2}{*}{$\{0,1\}$}&\hfil\multirow{2}{*}{$\{0,1\}$}&\hfil\multirow{2}{*}{$\{0,1\}$}&\hfil\multirow{2}{*}{$\{1\}$}\\
EN&&&&\\\hline
NE&\hfil\multirow{2}{*}{$\{0\}$}&\hfil\multirow{2}{*}{$\emptyset$}&\hfil\multirow{2}{*}{$\{0\}$}&\hfil\multirow{2}{*}{$\emptyset$}\\
NE&&&&\\\hline
NN&\hfil\multirow{2}{*}{$\{0,1\}$}&\hfil\multirow{2}{*}{$\{1\}$}&\hfil\multirow{2}{*}{$\{0,1\}$}&\hfil\multirow{2}{*}{$\{1\}$}\\
NE&&&&\\\hline
NN&\hfil\multirow{2}{*}{$\{0\}$}&\hfil\multirow{2}{*}{$\{0\}$}&\hfil\multirow{2}{*}{$\{0\}$}&\hfil\multirow{2}{*}{$\{0\}$}\\
NN&&&&\\\hline
\end{tabular}
\end{center}

Next we let $g$ be the function applied to $u_{i}$ by $\sigma_{M}$, where $i$ is in
\[\{3,\ldots, m+r\}.\]
The domain of $g$ will be $\mathcal{P}\times \{0,1\}$ and the output of $g$ will be in $\mathcal{P}$.
Assume that the input to $g$ is $(A,s)$.
First consider the case that $p_{i}=N$.
Then
\[g(A,0)=(A \cup \{a-1\colon a\in A\})\cap \{0,1,\ldots, n_{i}(Q)-n_{i}(P)\}\]
and
\[g(A,1)=A\cap \{0,1,\ldots, n_{i}(Q)-n_{i}(P)\}.\]
On the other hand, if $p_{i}=E$ then
\[g(A,0)=(A\cup \{a+1\colon a\in A\})\cap \{0,1,\ldots, n_{i}(Q)-n_{i}(P)\}\]
and
\[g(A,1)=\{a+1\colon a\in A\}\cap \{0,1,\ldots, n_{i}(Q)-n_{i}(P)\}.\]

Next we must describe the transition functions $\delta_{0}$ and $\delta_{2}$.
The former takes each pair $(\epsilon,s)$ to $\{s\}$, where $s$ is $0$ or $1$.
If $g$ is a function from $\{0,1\}\times \{0,1\}$ or $\mathcal{P}\times\{0,1\}$ to $\mathcal{P}$
then $\delta_{2}(g,j,k)$ is $g(j,k)$ whenever $(j,k)$ is in the domain of $g$.
This completes the description of the automaton \[A_{\lambda}=(\Sigma\cup \{(\epsilon,0),(\epsilon,1)\}, \mathcal{P}, \mathcal{P}-\{\emptyset\},\delta_{0},\delta_{2}).\]
Clearly $A_{\lambda}$ can be constructed in time bounded by $f(\lambda)$, where $f$ is
some function from the positive integers to the positive integers.

To complete the proof of \Cref{copeck}, we must show that $(T_{M},\sigma_{M},\varphi_{M})$
is a parse tree relative to the automaton $A_{\lambda}$.
We can accomplish this by proving our claim that the set applied to $u_{i}$ by $A_{\lambda}$ during its run on $\enc(T_{M}, \sigma_{M}, \varphi_{M}, Y)$ is the set $\{|N(L_{i})|-|N(p_{1}p_{2}\cdots p_{i})|\}$, where $L_{i}$ ranges over $\{E,N\}^{i}$ subject to $p_{1}p_{2}\cdots p_{i}\preccurlyeq L_{i}\preccurlyeq q_{1}q_{2}\cdots q_{i}$ and $Y\cap \{1,2,\ldots, i\}\subseteq N(L_{i})$.
When $i=2$, this is simply a matter of checking the table above, so we inductively assume the claim is true for $i-1$.
Let $A$ be the set applied to $u_{i-1}$.
If $p_{i}=N$ and $i$ is in $Y$, then the final character in $L_{i}$ must be $N$, so the set applied to $u_{i}$ should be exactly the numbers in $A$, as long as those numbers are not greater than $n_{i}(Q)-n_{i}(P)$.
But this set is exactly $g(A,1)$, where $g$ is the function applied to $u_{i}$ by the labelling $\sigma_{M}$.
So in this case, the set applied to $u_{i}$ is exactly as we claimed.

If $i$ is not in $Y$, then the final character in $L_{i}$ could be either $E$ or $N$, as long as the constraint $p_{1}p_{2}\cdots p_{i}\preccurlyeq L_{i}\preccurlyeq q_{1}q_{2}\cdots q_{i}$ is satisfied.
Thus the set applied to $u_{i}$ should contain all numbers of the form $a-1$ or $a$, where $a$ ranges over the members of $A$, as long as these numbers are in $\{0,1,\ldots, n_{i}(Q)-n_{i}(P)\}$.
But this is exactly the output $g(A,0)$.
The case when $p_{i}=E$ yields to exactly the same sort of analysis, so $(T_{M},\sigma_{M},\varphi_{M})$ is a parse tree relative to $A_{\lambda}$.
It is clear that $(T_{M},\sigma_{M},\varphi_{M})$ can be constructed in time bounded by $f(\lambda)(m+r)^{c}$ for some constant $c$, so the proof is complete.
\end{proof}

By applying \Cref{nature} to the previous result, we see that case (ii)
in \Cref{hammer} is proved.
Furthermore, we can deduce the following result.

\begin{corollary}
\label{dinner}
The class of lattice path matroids is pigeonhole.
\end{corollary}

\section{Frame matroids}
\label{frames}

Let $G$ be a graph with edge set $E$.
We allow $G$ to contain loops and parallel edges.
If $X$ is a subset of $E$, we use $G[X]$ to denote the subgraph
with edge set $X$ containing exactly those vertices that
are incident with an edge in $X$.
Similarly, if $N$ is a set of vertices, then $G[N]$ is the induced subgraph
of $G$ with $N$ as its vertex set.
A \emph{theta subgraph} consists of two distinct vertices joined by
three internally-disjoint paths.
A \emph{linear class} of cycles in $G$ is a family, \mcal{B}, of cycles such
that no theta subgraph of $G$ contains exactly two cycles in \mcal{B}.
Let \mcal{B} be a linear class of cycles in $G$.
A cycle in \mcal{B} is \emph{balanced}, and a cycle not in \mcal{B}
is \emph{unbalanced}.
A subgraph of $G$ is \emph{unbalanced} if it contains an
unbalanced cycle, and is otherwise \emph{balanced}.

\emph{Frame matroids} were introduced by Zaslavsky \cite{Zas91}.
The frame matroid, $M(G,\mcal{B})$, has $E$ as its ground set.
The circuits of $M(G,\mcal{B})$ are the edge sets of balanced cycles,
and the edge sets of minimal connected subgraphs
containing at least two unbalanced cycles, and no balanced cycles.
Such a subgraph is either a theta subgraph or a \emph{handcuff}.
A \emph{tight handcuff} contains two edge-disjoint cycles that have
exactly one vertex in common.
A \emph{loose handcuff} consists of two vertex-disjoint cycles
and a minimal path that joins the two cycles.
Note that if \mcal{B} contains every cycle, then $M(G,\mcal{B})$ is a graphic
matroid.
The set $X\subseteq E$ is independent in $M(G,\mcal{B})$
if and only if $G[X]$ contains no balanced cycle, and each
connected component of $G[X]$ contains at most one cycle.
The rank of $X$ in $M(G,\mcal{B})$ is the number of vertices in
$G[X]$, minus the number of balanced components of $G[X]$.

\begin{proposition}
\label{rouble}
Let $M=M(G,\mcal{B})$ be a $3$\dash connected frame matroid,
and let $(U,V)$ be a partition of the edge set of $G$ such that
$\lambda_{M}(U)\leq \lambda$.
There are at most $14\lambda-12$ vertices that are incident
with edges in both $U$ and $V$.
\end{proposition}

\begin{proof}
Let $n$ be the number of vertices in $G$.
We can assume that $G$ has no isolated vertices.
It then follows from the $3$\dash connectivity of $M$ that $G$ is connected.
Let $n_{U}$ and $n_{V}$ be the number of vertices in $G[U]$ and
$G[V]$, respectively.
Let $N$ be the set of vertices that are in both $G[U]$ and $G[V]$, so
$n+|N|=n_{U}+n_{V}$.
Each vertex in $N$ is incident with a connected component
of $G[U]$ and with a connected component of $G[V]$.
Since $G$ is connected, each component of
$G[U]$ or $G[V]$ contains at least one vertex of $N$.
Thus the connected components of $G[U]$ induce a partition of $N$.
There are no coloops in $M$, and it follows that if a component of $G[U]$ contains
only a single, non-loop, edge, then that edge joins two vertices of $N$.
Let $a$ be the number of such components.
Next we claim that if $X$ is a connected component of $G[U]$
such that $X$ is balanced and contains at least two edges, then
$X$ contains at least three vertices of $N$.
If this is not true, then we can easily verify that
$M$ has a $1$\dash\ or $2$\dash separation,
contradicting the hypotheses of the \namecref{nibble}.
Assume that there are $b$ balanced components of $G[U]$
with more than one edge, and let $\alpha_{i},\ldots, \alpha_{b}$
be the numbers of vertices these components share with $N$.
Our claim shows that $\alpha_{i}\geq 3$ for each $i$.
Finally, assume there are $c$ unbalanced components in $G[U]$,
and these components intersect $N$ in $\beta_{1},\ldots, \beta_{c}$ vertices,
respectively.
Thus $|N|=2a+\sum \alpha_{i}+\sum \beta_{i}$, and
$r_{M}(U) = n_{U}-(a+b)$.

Let $x$ be the number of components of $G[V]$
consisting of a single non-loop edge.
Assume there are $y$ balanced components of $G[V]$
with more than one edge, and that these intersect $N$ in
$\gamma_{1},\ldots, \gamma_{y}$ vertices.
Let $z$ be the number of unbalanced components of
$G[V]$, and assume that they intersect $N$ in
$\delta_{1},\ldots, \delta_{z}$ vertices, respectively.
So we have $|\gamma_{i}|\geq 3$,
$|N|=2x+\sum \gamma_{i}+\sum \delta_{i}$, and
$r_{M}(V) = n_{V}-(x+y)$.
Because $G$ is connected, $r(M)\geq n-1$,
and $r(M)=n-1$ if and only if $G$ is balanced.
Now we observe that
\begin{linenomath*}
\begin{multline*}
\lambda\geq r_{M}(U)+r_{M}(V)-r(M)
\geq n_{U}+n_{V}-(a+b+x+y)-(n-1)\\
=|N|-(a+b+x+y)+1.
\end{multline*}
\end{linenomath*}
This last quantity is equal to
$a+\sum \alpha_{i} + \sum \beta_{i} -(b+x+y)+1$, and also to
$x+\sum \gamma_{i} +\sum \delta_{i} - (a+b+y)+1$,
so both are at most $\lambda$.
By adding the two inequalities together, we obtain
\[
2\lambda\geq \sum \alpha_{i}+\sum \beta_{i}+\sum \gamma_{i}+\sum \delta_{i}-2(b+y)+2.
\]
But because each $\alpha_{i}$ is at least three, we also have
$b\leq \frac{1}{3}\sum \alpha_{i}$, and symmetrically
$y\leq \frac{1}{3}\sum \gamma_{i}$.
Therefore
\begin{equation}
\label{eqn1}
6(\lambda-1) \geq \sum \alpha_{i}+3\sum \beta_{i}+\sum \gamma_{i}+3\sum \delta_{i}.
\end{equation}

The edges counted by $a$ form a matching.
Therefore they are an independent set in $M$.
As $r_{M}(U)+r_{M}(V)-r(M)\leq \lambda$,
submodularity tells us that the intersection of
$\cl_{M}(U)$ and $\cl_{M}(V)$ has rank at most $\lambda$.
Thus there are at least $a-\lambda$ components of $G[U]$ that
consist of a single, non-loop, edge that is not in $\cl_{M}(V)$.
No such edge can be incident with one of the components of
$G[V]$ counted by $x$, for this would mean that a vertex of $G$
has degree equal to two, implying that $M$ contains a series pair.
This is impossible, since $M$ is $3$\dash connected (and we can obviously assume that
it has more than three elements).
Nor can such an edge join two vertices counted by the variables
$\delta_{1},\ldots,\delta_{z}$, for then the edge joins two
components of $G[V]$ that contain unbalanced cycles.
This means that the edge is in a handcuff, and hence in
$\cl_{M}(V)$.
Now we conclude that each of the (at least) $a-\lambda$
edges is incident with at least one vertex counted
by the variables $\gamma_{1},\ldots, \gamma_{y}$.
As the edges counted by $a$ form a matching, we now see that
$a-\lambda \leq \sum \gamma_{i}$.
We conclude that
\begin{multline*}
|N|=2a+\sum \alpha_{i} +\sum \beta_{i}
\leq 2\sum \gamma_{i} + 2\lambda +\sum \alpha_{i}+\sum \beta_{i}\\
\leq 2\sum \alpha_{i}+6\sum\beta_{i} +2\sum\gamma_{i} +2\lambda.
\end{multline*}
But \eqref{eqn1} implies that
$2\sum \alpha_{i}+6\sum\beta_{i} +2\sum\gamma_{i}\leq 12\lambda -12$, and the result follows.
\end{proof}

\begin{remark}
\label{adagio}
If we remove the constraint of $3$\dash connectivity from
\Cref{rouble}, then no bound on the number of vertices in both
$G[U]$ and $G[V]$ is possible.
To see this, let $c_{0},\ldots, c_{2n-1}$ be vertices in a cycle
of the $2$\dash connected graph $G$.
Assume that $G-\{c_{i},c_{j}\}$ is disconnected for any $i\ne j$.
Define \mcal{B} to be the family of cycles that contain
all of the vertices $c_{0},\ldots, c_{2n-1}$.
It is easy to verify that \mcal{B} is a linear class.
For any $i$ (modulo $2n$) let $C_{i}$ be the set of
edges contained in a path from $c_{i}$ to $c_{i+1}$ containing
no other vertex in $c_{0},\ldots, c_{2n-1}$.
Then
\[(C_{0}\cup C_{2}\cup\cdots\cup C_{2n-2},
C_{1}\cup C_{3}\cup\cdots\cup C_{2n-1})\]
is a $1$\dash separation of $M(G,\mcal{B})$, but obviously there
is no bound on the number of vertices incident with edges in both
sides of this separation.
\end{remark}

We will concentrate on two subclasses of frame matroids.
\emph{Bicircular} matroids are those frame matroids
arising from linear classes that contain only loops.
Thus every cycle with more than one edge is unbalanced.
For any graph, $G$, we define $B(G)$ to be the
bicircular matroid $M(G,\emptyset)$.
Thus every bicircular matroid is equal to
$B(G)\oplus U_{0,t}$ for some graph $G$ and some integer $t$.
Bicircular matroids can also be characterised as the
transversal matroids represented by systems of the form
$(A_{1},\ldots, A_{r})$, where each element of the ground set is in
at most two of the sets $A_{1},\ldots, A_{r}$.

Next we define gain-graphic matroids.
Again, we let $G$ be an undirected graph with edge set $E$
and (possibly) loops and multiple edges.
Define $A(G)$ to be
\begin{linenomath}
\begin{multline*}
\{(e,u,v)\colon e\ \text{is a non-loop edge joining vertices}\ u\ \text{and}\ v\}\\
\cup \{(e,u,u)\colon e\ \text{is a loop incident with the vertex}\ u\}.
\end{multline*}
\end{linenomath}
A \emph{gain function}, $\sigma$, takes $A(G)$ to a group $H$
and satisfies $\sigma(e,u,v)=\sigma(e,v,u)^{-1}$
for any non-loop edge $e$ with end-vertices $u$ and $v$.
If $W=v_{0}e_{0}v_{1}e_{1}\cdots e_{t}v_{t+1}$ is a walk of $G$, then the
\emph{gain-value} of $W$ is
$\sigma(W)=\sigma(e_{0},v_{0},v_{1})\cdots\sigma(e_{t},v_{t},v_{t+1})$.
Now let $C=v_{0}e_{0}v_{1}e_{1}\cdots e_{t}v_{t+1}$ be a cycle of $G$,
where $v_{0}=v_{t+1}$, and the other vertices are pairwise distinct.
Note that $\sigma(C)$ may depend on the choice of orientation of $C$ and if $H$ is nonabelian, it may also depend on the choice of starting
vertex.
However, if $\sigma(C)$ is equal to the identity, then this equality will
hold no matter which starting vertex and orientation we choose.
We declare a cycle to be balanced exactly when $\sigma(C)$ is
equal to the identity, and this gives rise to a linear class.
If \mcal{B} is such a linear class, then $M(G,\mcal{B})$ is
an \emph{\hgg} matroid.
Gain-graphic matroids play an important role in the works by
Kahn and Kung \cite{KK82}, and Geelen, Gerards, and
Whittle \cite{GGW13}.

Let $u$ be a vertex of $G$, and let $\alpha$ be an element of $H$.
The gain function $\sigma_{u,\alpha}$ is defined to be identical to
$\sigma$ on any loop and on any edge not incident with $u$.
Furthermore $\sigma_{u,\alpha}(e,u,v)=\alpha\sigma(e,u,v)$
when $e$ is a non-loop edge joining $u$ to a vertex $v$, and
in this case $\sigma_{u,\alpha}(e,v,u)$ is defined to be
$\sigma(e,v,u)\alpha^{-1}$.
The operation that produces $\sigma_{u,\alpha}$ from $\sigma$ is
called \emph{switching}.
Two gain functions that are related by switching have exactly the same
balanced cycles \cite[Lemma 5.2]{Zas89}.

The next \namecref{nibble} treats bicircular matroids
and gain-graphic matroids
simultaneously, since the arguments are essentially
identical.

\begin{theorem}
\label{nibble}
The class of $3$\dash connected bicircular matroids is efficiently
pigeonhole.
If $H$ is a finite group,
then the class of $3$\dash connected \hgg\ matroids
is efficiently pigeonhole.
\end{theorem}

\begin{proof}
Let $M$ be a $3$\dash connected matroid that is either bicircular or \hgg.
If $M$ is bicircular then it is succinctly represented by a description of a graph and a list of the balanced loops.
An \hgg\ matroid is described via a graph and a labelling
that assigns an element of $H$ to each orientation of an edge.
Let $G$ be the graph that represents $M$, so that $G$ is unlabelled
if $M$ is bicircular, and labelled if $M$ is \hgg.
We can assume that $G$ has no isolated vertices.

Let $(U,V)$ be a partition of $E$ such that $\lambda_{M}(U)\leq \lambda$
for some positive integer $\lambda$.
Let $N$ be the set of vertices that are in both $G[U]$ and $G[V]$,
so that $|N|\leq 14\lambda-12$ by \Cref{rouble}.
We will describe an equivalence relation $\approx_{U}$ on the subsets of $U$ and then show that $\approx_{U}$ can be computed in time bounded by $O(\pi(\lambda)|E(M)|^{c})$ for some function $\pi$ and some constant $c$.
Moreover, we will prove that $\approx_{U}$ satisfies conditions (i) and (ii) of \Cref{yakuza-II}.

Let $X$ be a subset of $U$.
If $M$ is \hgg\ then we can choose a maximal forest of $G[X]$ and then perform switching operations so that every edge in the forest is labelled with the identity of $H$ \cite[Lemma 5.3]{Zas89}.
If $M$ is bicircular, we choose the maximal forest but we do not need to perform any switchings.
Now $X$ is dependent in $M$ if and only if there is an edge not in the spanning forest that receives an identity label, or if there are two distinct edges not in the maximal forest that are in the same component of $G[X]$.
Thus we can test the independence of $X$ in polynomial time.

Let $X$ and $X'$ be subsets of $U$.
We consider the circumstances under which we declare $X$ and $X'$ to be equivalent under $\approx_{U}$.
Firstly, if $X$ and $X'$ are both dependent, then $X \approx_{U} X'$.
If exactly one of $X$ and $X'$ is dependent, then $X \not\approx_{U} X'$.
Now assume that both $X$ and $X'$ are independent.
In this case, any component of $G[X]$ or $G[X']$ contains at most one cycle.
Let $u$ and $v$ be vertices in a component of $G[X]$ or $G[X']$.
Let $\Gamma$ be this component.
We claim there are at most two paths of $\Gamma$ that join $u$ to $v$.
This is clear if $\Gamma$ is a tree, so assume that $\Gamma$ contains exactly one cycle.
Let $e$ be an edge of $\Gamma$ such that $\Gamma\backslash e$ is a tree.
If there are three distinct paths from $u$ to $v$ in $\Gamma$, then two of them use $e$.
Now there must be two distinct paths of $\Gamma\backslash e$ from either $u$ or $v$ to an end-vertex of $e$.
This is impossible so our claim is proved.

Let $\mathcal{X}$ and $\mathcal{X}'$ be the sets of connected components in $G[X]$ and $G[X']$, respectively, that have non-empty intersection with $N$.
In the case that $X$ and $X'$ are both independent, we declare that $X\approx_{U} X'$ if there is a bijection $\theta\colon \mathcal{X}\to\mathcal{X'}$ such that the following statements hold for every $\Gamma\in\mathcal{X}$,
\begin{enumerate}[label = \textup{(\roman*)}]
\item $\Gamma\cap N= \theta(\Gamma)\cap N$,
\item $\Gamma$ contains a cycle if and only if $\theta(\Gamma)$ contains a cycle, and
\item in the case that $M$ is \hgg\ and $u$ and $v$ are vertices of $\Gamma\cap N$, there is a path of $\Gamma$ from $u$ to $v$ with gain-value $h$ if and only if there is a path of $\theta(\Gamma)$ from $u$ to $v$ with gain-value $h$.
\end{enumerate}

It is clear that $\approx_{U}$ is an equivalence relation.
Next we count the equivalence classes.
Let $\tau(|N|)$ be the number of partitions of $N$ where at most one block of the partition is allowed to be empty.
We will think of the non-empty blocks in this partition as being the intersections of components of $G[X]$ with $N$.
Thus $\tau(|N|)$ counts the number of possible such intersections.
Note that $\tau$ is a non-decreasing function on the integers.

To choose an equivalence class of $\approx_{U}$, we first choose a non-empty collection of pairwise disjoint subsets of $N$.
Since the size of $N$ is at most $14\lambda-12$, we can do this in at most $\tau(14\lambda-12)$ ways.
Next we choose whether each of these components has a cycle or not.
The number of components that intersect $N$ is at most $14\lambda-12$, so we can make this choice in at most $2^{14\lambda-12}$ ways.
Finally, for each pair $(u,v)\in N\times N$, we choose at most two gain-values in $H$ for paths from $u$ to $v$.
The number of ways we can make this choice is at most $|H|^{2(14\lambda-12)^{2}}$.
Thus the number of equivalence classes under $\approx$ is at most
\[
\tau(14\lambda-12)2^{14\lambda-12}|H|^{2(14\lambda-12)^{2}}.
\]
Let this number be denoted by $\pi(\lambda)$.
Since $H$ is fixed, $\pi(\lambda)$ depends only on $\lambda$.
It is clear that we can test the equivalence $X\approx_{U} X'$ in time bounded by $\pi(\lambda)|E(M)|^{c}$ for some constant $c$.
Now we can complete the proof of \Cref{nibble} by showing that $\approx_{U}$ refines $\sim_{U}$.

To this end, assume that $X$ and $X'$ are independent subsets of $U$ and that $X\approx_{U} X'$.
Assume that $X\cup Z$ is dependent for
some $Z\subseteq V$, and let $C$ be a circuit of $M$ contained in $X\cup Z$.
We will prove that $X'\cup Z$ is also dependent in $M$ and this will complete the proof.

First assume that $C$ is a balanced cycle.
If $C$ is a balanced loop, then it is contained in $Z$, since $X$ is independent.
In this case $X'\cup Z$ is dependent and we have nothing left to prove.
Therefore we assume that $C$ is a balanced cycle with more than one edge, so $M$ is an \hgg\ matroid.
Now each component of $G[X\cap C]$ is a path, $P$, between two vertices of $N$.
We can replace each such $P$ with a path $P'$ of $G[X']$ running between the same vertices.
Moreover, since  $X\approx X'$ holds, we can choose $P'$ so that $\sigma(P')=\sigma(P)$.
When we perform all these substitutions on $C$ we obtain a walk $W$ in $G[X'\cup Z]$, where $\sigma(W)$ is the identity of $H$.
It is now easy to prove that $G[W]$ contains either a balanced cycle, or two distinct cycles.
In either case $G[X'\cup Z]$ is dependent in $M$ so we are done.

Now we assume that $C$ is a theta subgraph or a handcuff.
Let $\Gamma_{1},\ldots, \Gamma_{n}$ be the connected components of $G[X]$ that
contain edges of $C$.
For each $i$ let $D_{i}$ be the set of edges of $C$ contained in $\Gamma_{i}$.
Note that $D_{1},\ldots, D_{n}$ are pairwise disjoint sets of edges.
Because $X\approx X'$ we can make a choice of $D_{i}'$, a minimal set of edges in $X'$
for each $i$ such that the following conditions hold.
\begin{enumerate}[label = \textup{(\roman*)}]
\item $G[D_{i}']$ is connected,
\item every vertex of $G[D_{i}]\cap N$ is in $G[D_{i}']$, and
\item $G[D_{i}']$ contains a cycle if and only if $G[D_{i}]$ contains a cycle.
\end{enumerate}
Let $C_{0}$ be $C$.
For each $i$ let $C_{i}$ be the subgraph obtained from
$C_{i-1}$ by replacing the edges of $D_{i}$ with $D_{i}'$.
Thus $C_{n}$ is a subgraph of $G[X'\cup Z]$.
It is clear that each $C_{i}$ is connected.
We will show that $C_{n}$ contains at least two cycles, and then we will be done.

For any graph, $\Gamma$, let $\nu(\Gamma)$ be $|E(\Gamma)|-|V(\Gamma)|$.
If $\Gamma$ is connected, then $\nu(\Gamma)\geq-1$.
If $\Gamma$ is connected and contains exactly one cycle, then
$\nu(\Gamma)=0$.
Let $(L,R)$ be a partition of $E(\Gamma)$,
and assume that $\gamma$ vertices are incident with edges in both
$L$ and $R$.
It is easy to confirm that
\begin{equation}
\label{eqn2}
\nu(\Gamma)=\nu(\Gamma[L])+\nu(\Gamma[R])+\gamma.
\end{equation}

Note that $\nu(C_{0})=1$.
We assume inductively that $\nu(C_{i-1})\geq 1$.
Note that $G[D_{i}]$ may not be connected, but $G[D_{i}']$ is connected.
Our choice of $D_{i}'$ means that $\nu(D_{i}') \geq \nu(D_{i})$.
Furthermore, $G[D_{i}']$ has at least as many vertices in common with
$G[C_{i-1}-D_{i}]$ as $G[D_{i}]$ does.
It now follows from \eqref{eqn2} that $\nu(C_{i})\geq \nu(C_{i-1})\geq 1$.
Thus $\nu(C_{n})\geq 1$, and since $C_{n}$ is connected, it follows that
$C_{n}$ contains at least two cycles, as required.
\end{proof}

\begin{corollary}
\label{lichen}
Let \mcal{M} be the class of bicircular or
\hgg\ matroids (with $H$ a finite group).
Let $\psi$ be any sentence in \cmso.
We can test whether matroids in \mcal{M} satisfy $\psi$ using an 
algorithm that is fixed-parameter tractable with respect to branch-width.
\end{corollary}

\begin{proof}
This will follow immediately from \cite[Theorem 6.7]{FMN-I} and \Cref{nibble}
if we show that the succinct representations of bicircular
and \hgg\ matroids are minor-compatible.
We rely on \cite[Corollary 5.5]{Zas89} and \cite[Theorem 2.5]{Zas91}.
Let $M$ be a bicircular or \hgg\ matroid corresponding to the
graph $G$, and let $e$ be an edge of $G$.
Then $M\ba e$ is bicircular or \hgg, and corresponds to
$G\ba e$.
(In the case that $M$ is \hgg, the edge-labels in $G\ba e$ are
inherited from $G$.)

Contraction is somewhat more technical.
If $e$ is a non-loop, then
we first perform a switching (in the \hgg\ case)
so that the gain-value on $e$ is the identity.
We then simply contract $e$ from $G$.
The resulting labelled graph represents $M/e$.
Now assume $e$ is a loop of $G$ incident with the vertex $u$.
If $e$ is a balanced loop, we simply delete $e$, so now assume that $e$
is an unbalanced loop.
In the \hgg\ case, this implies that $H$ is non-trivial.
We obtain the graph $G'$ by deleting $u$ and
replacing each non-loop edge, $e'$, incident with $u$
with a loop incident with the other end-vertex of $e'$.
In the \hgg\ case, the loop $e'$ is labelled with any non-identity element.
Any other loops of $G$ that are incident with $u$ are added as balanced loops after contracting $e$.

It is clear that the operations of deletion and contraction can be
performed in polynomial time, so the classes of
bicircular  and \hgg\ matroids have
minor-compatible succinct representations as desired.
\end{proof}

\Cref{lichen} completes the proof of \Cref{hammer}.

\begin{remark}
\label{fodder}
Hlin\v{e}n\'{y} has shown \cite[p.~348]{Hli06c} that his work provides an
alternative proof of Courcelle's Theorem.
We can provide a simple new proof by relying on \Cref{lichen} and using
bicircular matroids as models for graphs.
We now briefly explain this strategy.

Let $\psi$ be a sentence in the counting monadic
second-order logic, \cmstwo\ of graphs.
This means that we can quantify over variables representing vertices,
edges,  sets of vertices and set of edges.
We have binary predicates for set membership, and also
an incidence predicate, which
allows us to express that an edge is incident with a vertex.
Furthermore, we have predicates which allow us to assert that a
set has cardinality $p$ modulo $q$, for any appropriate choice of
$p$ and $q$.
We need to show that there is a fixed-parameter tractable
algorithm for testing $\psi$ in graphs, with respect to the
parameter of tree-width.

Let $G$ be a graph, and let $G^{\circ}$ be the
graph obtained from $G$ by adding two loops at every vertex.
We need to interpret $\psi$ as a sentence about bicircular
matroids of the form $B(G^{\circ})$.
We let $\vertex(X_{i})$ be the \cmso\ formula stating that
$X_{i}$ is a $2$\dash element circuit.
Similarly, we let $\edge(X_{i})$ be a formula expressing that
$X_{i}$ is a singleton set not contained in a $2$\dash element circuit.
Now we make the following interpretations in $\psi$:
if $v$ is a vertex variable, we replace $\exists v$ with
$\exists X_{v} \vertex(X_{v})\ \land$, and we replace
$\forall v$ with $\forall X_{v} \vertex(X_{v})\ \to$.
We perform a similar replacement for variables representing
edges.
If $V$ is a variable representing a set of vertices, we replace $\exists V$ with
\[
\exists X \forall X_{1} (\sing(X_{1}) \land X_{1}\subseteq X)\to\\
\exists X_{2} (X_{1}\subseteq X_{2}\land X_{2}\subseteq X\land \vertex(X_{2})) \land
\]
where $\sing(X_{1})$ is a predicate expressing that $X_{1}$ contains exactly
one element.
There are similar replacements for variables representing sets of edges and for universal quantifiers.
Finally, we replace any occurrence of the predicate stating that
$e$ is incident with $v$ with a \cmso\ formula saying that
there is a $3$\dash element circuit that contains $X_{e}$ and one
of the elements in $X_{v}$.
We let $\psi'$ be the sentence we obtain by making these
substitutions.
It is clear that a graph, $G$, satisfies $\psi$ if and only if
$B(G^{\circ})$ satisfies $\psi'$.
Therefore \Cref{lichen} implies that there is a
fixed-parameter tractable algorithm for testing whether
$\psi'$ holds in matroids of the form $B(G^{\circ})$,
with respect to the parameter of branch-width.

To find the branch-width of a graph with edge set $E$, we consider a subcubic
tree, $T$, and a bijection from $E$ to the leaves of $T$.
If $(U,V)$ is a partition of $E$ displayed by an edge, $e$, of $T$, then we count
the vertices incident with edges in both $U$ and $V$.
This gives us the \emph{width} of $e$, and the maximum width of an
edge of $T$ is the width of the decomposition.
The lowest width across all such decompositions is the branch-width of the graph.
It is not difficult to see that the branch-width of the matroid $B(G^{\circ})$
is bounded by a function of the branch-width of the graph $G$,
and similarly the branch-width of $G$ is bounded by a function of the
branch-width of $B(G^{\circ})$.
But exactly the same relation holds between the branch-width and the
tree-width of $G$ \cite[(5.1)]{RS91}.
Now it follows that there is a fixed-parameter tractable algorithm for
testing whether $\psi$ holds in graphs, where the parameter is tree-width.
This proves Courcelle's Theorem \cite{Cou90}.
\end{remark}

When $H$ is not finite, the class of \hgg\ matroids is
not even pigeonhole, as we now show.
First we require the following \namecref{raunch}.

\begin{proposition}
\label{raunch}
Let $H$ be an infinite group, and let $m$ and $n$
be positive integers.
There are disjoint subsets $A,B\subseteq H$ such that
$|A|=m$, $|B|=n$, and
$\{ab\colon a\in A, b\in B\}$ is disjoint from $A\cup B$
and has cardinality $mn$.
\end{proposition}

\begin{proof}
Assume that $m=1$.
Choose $B$, an arbitrary subset of $n$ elements
that does not include the identity.
The cancellation rule implies the result if we let $A$ be a
singleton set containing an element
not in $B\cup \{b_{1}b_{2}^{-1}\colon b_{1},b_{2}\in B\}$.
The result similarly holds if $n=1$.
Now we let $m$ and $n$ be chosen so that
$m+n$ is as small as possible
with respect to the \namecref{raunch} failing.
Let $A'$ and $B$ be disjoint subsets such that
$|A'|=m-1$, $|B|=n$, and
$\{ab\colon (a,b)\in A'\times B\}$ has cardinality $(m-1)n$
and is disjoint from $A'$ and $B$.
We choose an element $x$
not in $A'\cup B$ that does not belong to
$\{ab^{-1}\colon a\in A, b\in B\}$, nor to
$\{b_{1}b_{2}^{-1}\colon b_{1},b_{2}\in B\}$, nor to
$\{ab_{1}b_{2}^{-1}\colon a\in A,b_{1},b_{2}\in B\}$.
Now we simply let $A$ be $A\cup \{x\}$.
\end{proof}

\begin{proposition}
\label{object}
Let $H$ be an infinite group.
There are rank\dash $3$ \hgg\ matroids with arbitrarily high
decomposition-width.
Hence the class of \hgg\ matroids is not pigeonhole.
\end{proposition}

\begin{proof}
Assume otherwise, and let $K$ be an
integer such that $\dw(M)\leq K$ whenever $M$ is
a rank\dash $3$ \hgg\ matroid.

Zn\'{a}m proved that if a bipartite graph with
$n$ vertices in each side of its bipartition has more than
$(d-1)^{1/d}n^{2-1/d}+n(d-1)/2$ edges, then it has
a subgraph isomorphic to $K_{d,d}$ \cite{Zna63}.
Choose an integer $d$ such that $d^{2}>K$.
Choose the integer $p$ so that
\[
\frac{1}{2}p^{2}>(d-1)^{1/d}p^{2-1/d}+\frac{1}{2}p(d-1).
\]
Finally, choose the integer $q$ such that $q-p\geq q/2\geq p$
and
\[
\frac{1}{3}(q^{2}+2q)-p(2q - p + 2)
>(d-1)^{1/d}(q-p)^{2-1/d}+\frac{1}{2}(q-p)(d-1).
\]

Using \Cref{raunch}, we choose disjoint subsets
$A=\{a_{1},\ldots, a_{q}\}$ and
$B=\{b_{1},\ldots, b_{q}\}$ of $H$
such that $a_{i}b_{j}\ne a_{p}b_{q}$ whenever
$(i,j)\ne (p,q)$.
Let $AB$ be $\{a_{i}b_{j}\colon 1\leq i,j\leq q\}$.
We can also assume that $AB$ is disjoint from $A\cup B$.
Let $G$ be a graph on vertex set $\{v_{1},v_{2},v_{3}\}$,
where there are $q$ parallel edges between
$v_{1}$ and $v_{2}$ and between $v_{2}$ and $v_{3}$,
and $q^{2}$ parallel edges between $v_{1}$ and $v_{3}$.
We let $\sigma$ be the gain function applying the
elements in $A$ to the $q$ arcs from
$v_{1}$ to $v_{2}$,
the elements in $B$ to the arcs from
$v_{2}$ to $v_{3}$, and the elements in $AB$
to those arcs from $v_{1}$ to $v_{3}$.
We identify these group elements with the ground set
of the \hgg\ matroid
$M=M(G,\sigma)$.
Therefore $M$ is a rank\dash $3$ matroid with ground set
$A\cup B\cup AB$.
Its non-spanning circuits are the $3$\dash element
subsets of $A$, $B$, or $AB$,
along with any set of the form
$\{a_{i},b_{j},a_{i}b_{j}\}$.

Let $(T,\varphi)$ be a decomposition of $M$ with the
property that if $U$ is any displayed set, then
$\sim_{U}$ has at most $K$ equivalence classes.
As in the proof of \Cref{apogee}, we let $e$ be an edge
of $T$ such that each of the displayed sets, $U_{e}$ and $V_{e}$,
contains at least $|E(M)|/3=(q^{2}+2q)/3$ elements.
We construct a complete bipartite graph with vertex set
$A\cup B$ and edge set $AB$,
where $a_{i}b_{j}$ joins $a_{i}$ to $b_{j}$.
We colour a vertex or edge red if it belongs to $U_{e}$,
and blue otherwise.
Without loss of generality, we will assume that
at least $q/2\geq p$ vertices in $A$ are red.

Assume that $B$ contains at least $p$ blue vertices.
We choose $p$ such vertices, and $p$ red vertices from $A$,
and let $G'$ be the graph induced by these $2p$ vertices.
There are $p^{2}$ edges in $G'$.
Assume that at least $p^{2}/2$ of them are red
(the case that at least $p^{2}/2$ of them are blue is
almost identical).
Our choice of $p$ means that $G'$ contains a
subgraph isomorphic to $K_{d,d}$ consisting of red edges.
Thus there are elements
$a_{i_{1}},\ldots, a_{i_{d}}\in A\cap U_{e}$ and
$b_{j_{1}},\ldots, b_{j_{d}}\in B\cap V_{e}$ such that
every element $a_{i_{p}}b_{j_{q}}$ is in $U_{e}$.
For $(l,k)\ne (p,q)$, we see that
$\{a_{i_{l}},a_{i_{l}}b_{j_{k}}\}$ is not equivalent
to
$\{a_{i_{p}},a_{i_{p}}b_{j_{q}}\}$, since
$\{a_{i_{l}},a_{i_{l}}b_{j_{k}},b_{j_{k}}\}$ is a circuit of $M$,
and $\{a_{i_{p}},a_{i_{p}}b_{j_{q}},b_{j_{k}}\}$ is a basis.
Therefore $\sim_{U_{e}}$ has at least
$d^{2}>K$ equivalence classes, and we have a contradiction.
We must now assume that $B$ contains fewer than
$p$ blue vertices, and hence at least $q-p\geq q/2$ red vertices.
Thus a symmetrical argument shows that $A$ contains fewer than
$p$ blue vertices.

We choose $q-p$ red vertices from each of $A$ and $B$, and
let $G''$ be the subgraph induced by these vertices.
Let $g$ stand for the number of blue edges in $G''$.
The number of edges not in $G''$ is equal to
$q^{2}-(q-p)^{2}=2pq-p^{2}$.
As there are $g$ blue edges in $G''$,
at most $2pq-p^{2}$ blue edges not in $G''$, and
fewer than $2p$ blue vertices, it follows that
$|V_{e}| < g+2pq - p^{2}+ 2p$.
Since
$(q^{2}+2q)/3\leq |V_{e}|$, we deduce that
\[
\frac{1}{3}(q^{2}+2q) - p(2q - p + 2) <g.
\]
Our choice of $q$ now means that $G''$ has a
subgraph isomorphic to $K_{d,d}$ consisting of
blue edges.
Thus we have elements
$a_{i_{1}},\ldots, a_{i_{d}}\in A\cap U_{e}$ and
$b_{j_{1}},\ldots, b_{j_{d}}\in B\cap U_{e}$ such that
$a_{i_{p}}b_{j_{q}}$ is in $V_{e}$ for each $p$ and $q$.
For $(l,k)\ne (p,q)$, we see that
$\{a_{i_{l}},b_{j_{k}}, a_{i_{l}}b_{j_{k}}\}$ is a circuit of $M$,
while $\{a_{i_{p}},b_{j_{q}},a_{i_{l}}b_{j_{k}}\}$ is a basis.
This implies there are at least $d^{2}$ equivalence classes under
$\sim_{U_{e}}$, so we again have a contradiction.
\end{proof}

\section{Open problems}

We have proved that the class of lattice path matroids is
pigeonhole, but we have not yet proved that it is strongly pigeonhole.
Nevertheless, we believe this to be the case.

\begin{conjecture}
\label{ulster}
The class of lattice path matroids is efficiently pigeonhole.
\end{conjecture}

The classes of fundamental transversal matroids and lattice path
matroids are both closed under duality
(\cite[Proposition 11.2.28]{Oxl11} and \cite[Theorem 3.5]{BdMN03}).
Thus they belong to the intersection of transversal and
cotransversal matroids.
We suspect that \Cref{yuppie}
(and \Cref{ulster}) exemplify a more general result.

\begin{conjecture}
\label{iguana}
The class of matroids that are both transversal and
cotransversal is strongly pigeonhole.
\end{conjecture}

Despite the existence of examples as in \Cref{adagio},
we firmly believe the next conjecture.

\begin{conjecture}
\label{hubbub}
The class of bicircular matroids is efficiently pigeonhole.
Let $H$ be a finite group.
The class of \hgg\ matroids is efficiently pigeonhole.
\end{conjecture}

\section{Acknowledgements}
We thank Geoff Whittle for several important conversations, and
Yves de Cornulier for suggesting a proof of \Cref{raunch}.
Funk and Mayhew were supported by a Rutherford Discovery Fellowship,
managed by Royal Society Te Ap\={a}rangi.
We also thank the referee for their helpful comments.



\end{document}